\documentclass[11pt,letterpaper]{amsart}

\usepackage{amsmath,amssymb,amsthm,amscd}
\usepackage[shortlabels]{enumitem}
\usepackage{mathtools}
\usepackage{color}
\usepackage{bbm}
\usepackage{tikz-cd}

\usepackage[hmargin=2.5cm,vmargin=2cm]{geometry}
\usepackage[normalem]{ulem}
\usepackage{tcolorbox}
\usepackage{mathabx}

\usepackage{hyperref}
\usepackage{float}
\hypersetup{
colorlinks=true,
linkcolor={blue},
citecolor={red!50!black}
}

\usepackage{comment}

\newcommand{\nospacepunct}[1]{\makebox[0pt][l]{\,#1}}

\newtheorem{theorem}{Theorem}[section]
\newtheorem*{theorem*}{Theorem}
\newtheorem{lemma}[theorem]{Lemma}
\newtheorem*{lemma*}{Lemma}
\newtheorem{corollary}[theorem]{Corollary}

\newtheorem{proposition}[theorem]{Proposition}
\newtheorem{definition}[theorem]{Definition}
\newtheorem{remark}[theorem]{Remark}

\newtheorem{examples}[theorem]{Examples}

\newtheorem{thm}{Theorem}[section]

\newtheorem{ques}{Question}[section]

\numberwithin{equation}{section}

\newcommand{\A}{\mathcal A}
\newcommand{\cB}{\mathcal B}

\newcommand{\cD}{\mathcal D}

\newcommand{\cF}{\mathcal F}

\newcommand{\cH}{\mathcal H}
\newcommand{\cI}{\mathcal I}
\newcommand{\cJ}{\mathcal J}

\newcommand{\cM}{\mathcal M}

\newcommand{\cR}{\mathcal R}

\newcommand{\cU}{\mathcal U}

\newcommand{\cX}{\mathcal X}

\def\Bz{\mathbb{B}}
\def\Cz{\mathbb{C}}
\def\Ez{\mathbb{E}}
\def\Fz{\mathbb{F}}

\def\Kz{\mathbb{K}}

\def\1z{\mathbb{1}}
\newcommand{\fC}{\mathfrak C}

\newcommand{\fG}{\mathfrak G}

\newcommand{\dom}{\textup{dom}}
\newcommand{\im}{\textup{im}}

\newcommand{\acts}{\text{\reflectbox{$\uptodownarrow$}}}
\newcommand{\coacts}{\downtouparrow}

\newcommand{\supp}{\textup{supp}}

\newcommand{\Ad}{\textup{Ad}}

\newcommand{\id}{\textup{id}}

\newcommand{\mfj}{\mathfrak{j}}

\newcommand{\env}{\textup{env}}
\newcommand{\Fin}{\text{Fin}}
\newcommand{\alg}{\textup{alg}}
\newcommand{\eps}{\varepsilon}
\newcommand{\spn}{\textup{span}}

\renewcommand{\max}{\textup{max}}
\newcommand{\mfd}{\mathfrak{d}}
\newcommand{\mft}{\mathfrak{t}}
\newcommand{\mfr}{\mathfrak{r}}
\newcommand{\mfs}{\mathfrak{s}}
\newcommand{\mff}{\mathfrak{f}}
\newcommand{\fix}{\textup{fix}}
\newcommand{\Env}{\textup{Env}}



\begin{document}

\title{Normal coactions extend to the C*-envelope}
\author[K.A. Brix]{Kevin Aguyar Brix}
\address[K.A. Brix]{Department of Mathematical Sciences, University of Southern Denmark, 5230 Odense, Denmark}
\email{kabrix.math@fastmail.com}

\author[C. Bruce]{Chris Bruce}
\address[C. Bruce]{School of Mathematics, Statistics and Physics, Herschel Building, Newcastle University, Newcastle upon Tyne, NE1 7RU, UK}
\email{chris.bruce@newcastle.ac.uk}

\author[A. Dor-On]{Adam Dor-On}
\address[A. Dor-On]{Department of Mathematics, University of Haifa, Mount Carmel, Haifa 3103301, Israel}
\email{adoron.math@gmail.com}

\thanks{K.A. Brix was supported by a DFF-international postdoc grant 1025-00004B. C. Bruce has received funding from the European Union’s Horizon 2020 research and innovation programme under the Marie Sklodowska-Curie grant agreement No 101022531. A. Dor-On was partially supported by an NSF / BSF grant no. 2530543 / 2023695 (respectively) and a DFG Middle-Eastern collaboration project no. 529300231. This research was partially supported by a European Research Council (ERC) under the European Union’s Horizon 2020 research and innovation programme 817597 and the London Mathematical Society through a Research in Pairs Grant 42204.}

\date{}
\maketitle

\begin{abstract}
We show that a normal coaction of a discrete group on an operator algebra extends to a normal coaction on the C*-envelope. This resolves an open problem attempted by several experts in the area, and provides a more direct proof of a prominent result of Sehnem. 
As an application, we resolve a question of X. Li, where we identify the C*-envelopes of the operator algebras of groupoid-embeddable categories and of cancellative right LCM monoids. This latter class includes many examples of monoids that are not group-embeddable.
\end{abstract}

\section{Introduction}
In this paper, we show that every normal coaction of a discrete group on an operator algebra extends to a coaction on the C*-envelope (Theorem~\ref{thm:A}). Our proof relies only on elementary principles from non-abelian duality theory for coactions on \emph{operator algebras} inspired by the theory of non-abelian duality for coactions on C*-algebras as developed by Echterhoff, Kaliszewski, Quigg, and Raeburn \cite{EKQR}, as well as duality theory for actions by abelian groups on operator algebras as developed by Katsoulis and Ramsey \cite{KR, KR18}, to name but a few.
Our work resolves an open problem considered by several experts in the area, and yields a shorter, conceptual proof of a prominent result of Sehnem. As a consequence of our results, the C*-envelope always coincides with the equivariant C*-envelope for a normal coaction (Corollary~\ref{cor:main}).

For an operator algebra arising from a cancellative small category, we use Theorem~\ref{thm:A} combined with a novel \'{e}tale groupoid-theoretic approach to find a sufficient condition for its C*-envelope to coincide with the boundary quotient C*-algebra of the category. This allows us to resolve a question of Li, and identify the C*-envelopes of operator algebras arising from groupoid-embeddable categories (e.g., all submonoids of groups) and all cancellative right LCM monoids with the boundary quotient C*-algebra of the category/monoid (Theorems~\ref{thm:B} and \ref{thm:C}).
We also provide a simple computation of the C*-envelope for all finitely aligned higher-rank graphs, even all $P$-graphs (Theorem~\ref{thm:Pgraphs}).

\subsection*{Context and main result}
In a series of seminal papers, Arveson introduced the C*-envelope of an operator algebra as a non-commutative generalization of the Shilov boundary from function theory \cite{Arv69,Arv72,Arv98}. The C*-envelope was shown to exist through the use of the injective envelope in the work of Hamana \cite{Ham79}. Today, there are dilation theoretic ways of defining the C*-envelope as a universal C*-algebra, which then aid in its computation. For instance, the C*-envelope can be defined as the universal C*-algebra with respect to boundary representations, as proven by Arveson in the separable case \cite{Arv08} and by Davidson--Kennedy in general \cite{DKe15}.

The C*-envelope $C_\env^*(\A)$ of an operator algebra $\A$ is the smallest C*-algebra containing a completely isometric copy of $\A$ (see \S~\ref{ss:coactions} for the precise definition).
The structure and invariants of the C*-envelope then provide a new lens for studying non-selfadjoint operator algebras via selfadjoint (i.e., C*-algebra) theory. Arveson's works provide a deep and fruitful connection between the theories of selfadjoint and non-selfadjoint operator algebras, and it is now well-established that structural results for C*-envelopes are crucial for discovering and making use of new invariants for the underlying operator algebra. Indeed, in \cite{DFK17, DK20, DKKLL, Seh22} non-self-adjoint techniques are used to establish gauge-compatible co-universal properties for C*-algebras, in \cite{KR, Katsoulis, DK20, KR:IMRN} the identification of the C*-envelope (or the coaction C*-envelope in \cite{DKKLL}) is used to resolve crossed product Hao--Ng isomorphism problems, and in \cite{DEG} the computation of the C*-envelope together with $K$-theory for C*-algebras are used to classify operator algebras arising from directed graphs up to different notions of isomorphisms. 

Beyond interactions with C*-algebra theory, C*-envelopes appear in relation to important open problems in non-commutative boundary theory. These include the Arveson--Douglas essential normality conjecture \cite{Arv05} and Arveson's hyperrigidity conjecture \cite{Arv11} (which was recently resolved by the third-named author together with Bilich in \cite{BD+}), with applications to the theory of analytic varieties and approximation theory of functions. With such potential applications in mind, Arveson initiated a program of finding explicit descriptions for C*-envelopes of naturally occurring non-selfadjoint operator algebras. A general context in which one can answer this question normally consists of a Toeplitz C*-algebra $\mathcal{T}$ considered as a C*-cover of a canonical norm-closed subalgebra $\mathcal{T}^+$ both generated by operators that comprise an analogue of a left regular representation. This non-selfadjoint operator algebra $\mathcal{T}^+$ is often called the \emph{tensor algebra} of the context, and a general scheme for computing its C*-envelope treats several classes of operator algebras simultaneously. 

Starting with the works of Muhly and Solel \cite{MS98}, followed by Fowler, Muhly, and Raeburn, \cite{FMR03}, it was shown for large classes of C*-correspondences that the C*-envelope of a tensor algebra $\mathcal{T}^+_X$ of a C*-correspondence $X$ is the Cuntz--Pimsner algebra $\mathcal{O}_X$ in the sense of Pimsner \cite{Pimsner97}. After Katsura's refinement of Pimsner's definition \cite{Kat04}, this result was shown to hold for general C*-correspondences by Katsoulis and Kribs \cite{KK06}.

A natural ground for further extension, unification, as well as insight was made in the work of Fowler on discrete product systems of C*-correspondences over monoids \cite{Fow02}. Fowler's Toeplitz algebras generalize left regular semigroup C*-algebras, as well as Toeplitz C*-algebras for C*-correspondences. Later, Sims and Yeend \cite{SY10} exhibited a candidate for the C*-envelope for a large class of tensor algebras in Fowler's context. This Sims--Yeend C*-algebra was later shown by Carlsen, Larsen, Sims, and Vitadello \cite{CLSV11} to satisfy an appropriate gauge-compatible co-universal property, implying natural gauge-invariant uniqueness theorems for these C*-algebra. The Sims--Yeend candidate was later put in a broader context by Sehnem \cite{Seh18}, who introduced strong covariance algebras for arbitrary product systems over group-embeddable monoids.

Nearly 20 years after Fowler's paper, the third-named author together with Katsoulis \cite{DK20} extended the result of Katsoulis and Kribs from \cite{KK06}. They computed the C*-envelope in Fowler's context when the monoid embeds in an \emph{abelian} lattice-ordered group, and showed that Seh\-nem's covariance algebra coincides with the C*-envelope in this setting.

Despite some developments, the lack of Pontryagin duality for non-abelian groups posed a significant challenge for determining the C*-envelope of tensor algebras in Fowler's context beyond monoids that embed in abelian groups (see \cite{DKKLL, KKLL22a,KKLL22b}).
This problem is intimately related to coactions on operator algebras. More precisely, all tensor algebras mentioned above come equipped with a canonical normal coaction of a discrete group (see \S~\ref{ss:coactions} for the precise definitions). It was observed in \cite{DK20} that the problem of computing the C*-envelope in Fowler's context can be abstracted to the more general problem of whether the canonical coaction on an operator algebra extends to its C*-envelope.

\begin{ques}
\label{ques:extends?}
Suppose there is a coaction of a discrete group $G$ on an operator algebra $\A$. Does this coaction extend to a coaction of $G$ on $C^*_{\env}(\A)$?
\end{ques}

Progress on this question in Fowler's context for normal coactions (e.g., all coactions by amenable groups) was made in \cite{KKLL22a,KKLL22b} under additional assumptions, and was eventually resolved completely for group-embeddable monoids by Sehnem in \cite{Seh22}. More precisely, Sehnem showed that the covariance algebra from \cite{Seh18} coincides with the C*-envelope for product systems over an arbitrary group-embeddable monoid, circumventing Question \ref{ques:extends?} for general normal coactions.
Sehnem's paper \cite{Seh22} was a major advance and marks the state-of-the-art, but the proofs require intricate technical arguments in the specific setting of product systems, leaving Question~\ref{ques:extends?} open in general (see \cite[Remark~5.2]{Seh22}). 

Duality theory for operator algebras has a long history in the context of C*-algebras, dating back to the work of Takesaki and Takai in the 1970s \cite{Tak73,Tak75}. Surprisingly, this theory has found fertile grounds only recently in the general theory of operator algebras with the work of Katsoulis and Ramsey \cite{KR}, where a version of Takai duality for general operator algebras was proven. This led to applications in establishing semisimplicity of certain non-selfadjoint operator algebras. Further applications of duality theory of operator algebras are found in \cite{KR18,EKR25}. Thus, duality theory is emerging as a powerful tool for the study of non-selfadjoint operator algebras. One of our main contributions in this paper is to showcase the strength of \emph{non-abelian} duality theory of operator algebras to resolve Question~\ref{ques:extends?} for all normal coactions.

\begin{thm}[Theorem~\ref{thm:main}]
\label{thm:A}
Let $\A$ be an operator algebra with a contractive self-adjoint approximate identity. Then, any normal coaction of a discrete group $G$ on $\A$ has a (necessarily unique) extension to a normal coaction of $G$ on $C_\env^*(\A)$. 
\end{thm}
By the co-universal property of the C*-envelope, every \emph{action} of a discrete group on an operator algebra extends to an action on the C*-envelope.
However, it is not at all clear whether a \emph{coaction} extends to the C*-envelope.
In order to overcome this problem, we extend Katayama duality \cite{Kat} to normal coactions of discrete groups on operator algebras (Theorem~\ref{thm:Katduality}). This allows us to embed the C*-envelope in a reduced double crossed product in such a way that the restriction of the canonical coaction on the double crossed product gives us the desired coaction. We provide an outline of the proof. 

Suppose $\delta\colon G \downtouparrow\A$ is a normal coaction of a discrete group $G$ on an operator algebra $\A$. To show that $C_\env^*(\A)$ carries a normal coaction of $G$ that extends $\delta$, we consider the embedding of $C_\env^*(\A)$ in $C^*_\env(\A) \otimes \Kz$ by tensoring with a rank-one projection on the second tensor factor, where $\Kz=\Kz(\ell^2(G))$. By Katayama duality for operator algebras (Theorem~\ref{thm:Katduality}), $\A \otimes\Kz$ is equivariantly isomorphic to a crossed product by $G$. Now by \cite[Theorem~2.5]{Katsoulis}, $C_\env^*(\A \otimes\Kz) \cong C^*_\env(\A)\otimes \Kz$ is also a reduced crossed product by $G$, and thus carries a canonical normal coaction of $G$. We then prove that $C_\env^*(\A)$ is invariant under this coaction, so that the restriction of this coaction to the image of $C_\env^*(\A)$ gives us a normal coaction on $C_\env^*(\A)$. We then verify on generators that this restriction is an extension of $\delta$. 

Our approach is concise, and leads to new avenues of research.
For instance, Theorem~\ref{thm:A} may lead to new applications in dilation theory,
and has important consequences for Arveson's program of computing C*-envelopes. 

First, in the special case of product systems over group-embeddable monoids, we show in Remark \ref{r:sehnems} how Theorem~\ref{thm:A} leads to a shorter, conceptually simpler proof of Sehnem's main result, \cite[Theorem~5.1]{Seh22} (thus also recovering results in \cite{KKLL22a, KKLL22b}). 
Second, for an operator algebra arising from a cancellative small category, we use Theorem~\ref{thm:A} to provide a sufficient condition for its C*-envelope to coincide with the boundary quotient C*-algebra of the category. We verify this sufficient condition in two important examples classes. In the case when the small category is groupoid-embeddable or in the case when the category is a cancellative (\emph{not} necessarily group-embeddable) right LCM monoid, which then yields a description of the C*-envelope, see Theorems~\ref{thm:B} and \ref{thm:C}. We emphasize that examples of operator algebras arising from cancellative small categories do not naturally fit into Fowler's context. On the other hand, examples of operator algebras that do fit in that context now enjoy a significantly more conceptual proof for computing their C*-envelope, avoiding the technical language of product systems.
As this application makes novel use of \'etale groupoid theory in the setting of non-selfadjoint operator algebras, we explain the context and consequences of our applications below.

\subsection*{Application to operator algebras generated by partial isometries}

Left cancellative small categories are studied as natural generalizations of monoids \cite{Doh15}, and have recently been used to resolve open problems on finiteness properties for Thompson-like groups and topological full groups \cite{Wit19, SWZ19, LiGarII}. 
C*-algebras of left cancellative small categories \cite{Spi14, Spi20, LiGarI} unify the theories of several classes of C*-algebras.
For instance, semigroup C*-algebras \cite{Co67, Li12, Li13, CELY, LS22}, which have connections to, e.g., algebraic number theory \cite{CDL, CEL, BL23};
Cuntz--Krieger algebras \cite{Cu77, CK80, Cu81} and (higher-rank) graph C*-algebras \cite{KPRR97, RS99, KP00, Rae05}, which led to the general theory of topological full groups \cite{Mat12,Mat15}
and have important connections to symbolic dynamics \cite{Kr80, Hu95, ERRS21, CRST21, CDE23}.

All C*-algebras mentioned above are generated by families of partial isometries such that the composition law (as operators on Hilbert space) can be interpreted as composing morphisms in a left cancellative small category. 
Such a category then provides an \emph{orientation} inside the C*-algebra, in the sense that it determines an operator algebra of partial isometries that need \emph{not} be closed under taking adjoints. We think of this operator algebra as an ``irreversible'' operator algebra inside the ambient C*-algebra.
This orientation corresponds to the ``one-sided'' or ``irreversible'' nature of the underlying dynamics, the direction in a directed graph, 
or a notion of positivity in a group. It was Spielberg who recently discovered this vastly general and unifying framework of left cancellative small categories \cite{Spi14, Spi20}.
See also \cite{OP, BKQ, LiGarI} for further developments.

A left cancellative small category $\fC$ determines a Toeplitz-type C*-algebra $C_\lambda^*(\fC)$ (Definition~\ref{def:Toeplitz}), a boundary quotient C*-algebra $\partial C_\lambda^*(\fC)$ (Definition~\ref{def:bdry}), and an operator algebra $\A_\lambda(\fC)\subseteq C_\lambda^*(\fC)$ (Definition~\ref{def:A(fC)}).
At the masterclass ``Dilation and classification in operator algebra theory'' at the University of Copenhagen in October 2022, Xin Li posed the following natural question.
\begin{ques}
\label{ques:Li}
 Given a left cancellative small category $\fC$, does the C*-envelope of $\A_\lambda(\fC)$ coincide with the boundary quotient C*-algebra $\partial C_\lambda^*(\fC)$?  
\end{ques}
 This question is known to have an affirmative answer for certain subclasses of left cancellative small categories, e.g., when $\fC$ is a higher-rank graph or a submonoid of a group, see \cite{DK20, Seh22}. However, new techniques are necessary to answer it for more general categories. 
 Our main application of Theorem~\ref{thm:A} is a sufficient condition for answering Question~\ref{ques:Li} that we verify for all groupoid-embeddable categories and all cancellative (not necessarily group-embeddable) right LCM monoids.

\begin{thm}[Theorem~\ref{thm:gpoid-embeddable}]
\label{thm:B}
Let $\fC$ be a subcategory of a groupoid $\fG$. Then, the C*-envelope of $\A_\lambda(\fC)$ is canonically *-isomorphic to the boundary quotient C*-algebra $\partial C_\lambda^*(\fC)$.
\end{thm}

The proof of this application uses novel \'{e}tale groupoid-theoretic techniques and consists of two parts. 
The first step is to show  that $\partial C_\lambda^*(\fC)$ is a C*-cover of $\A_\lambda(\fC)$ (Theorem~\ref{thm:ciAr}). 
Here, we rely heavily on inverse semigroups and \'{e}tale groupoids underlying the C*-algebras \cite{Exel:comb, ExelPardo, Spi14, Spi20, LiGarI}, and this marks a new interaction between the theory of non-selfadjoint operator algebras and \'{e}tale groupoids. Interestingly, this part requires only the much weaker assumption that the category be cancellative with a Hausdorff groupoid. Our approach also provides new proofs of several established results from \cite{KKLL22a} and \cite{Seh22}. 
We give an example of this by providing a short direct computation of the C*-envelope for finitely aligned higher-rank graphs and $P$-graphs for group-embeddable $P$ (Theorem~\ref{thm:Pgraphs}). 
In the second step, we show injectivity of the canonical map from $\partial C_\lambda^*(\fC)$ onto $C_\env^*(\A_\lambda(\fC))$ that arises from the co-universal property of the C*-envelope. 
This requires Theorem~\ref{thm:A} together with a careful analysis of the \'etale groupoid model underlying $\partial C_\lambda^*(\fC)$. 

As a byproduct of our analysis, we prove that when $\fC$ is a subcategory of a discrete groupoid, then $\partial C_\lambda^*(\fC)$ can be realized as a crossed product C*-algebra for a canonical partial action (Corollary~\ref{cor:idpure}). 
This is a nontrivial generalization of \cite[Proposition~3.10]{Li:IMRN} and is of independent interest.
We believe our approach using \'etale groupoids will have important consequences for future research, e.g., into new example classes of categories that are not groupoid-embeddable where the associated \'etale groupoids are not Hausdorff.

Our final application is to identify the C*-envelopes of the operator algebras of all cancellative right LCM monoids. Importantly, this class contains Dehornoy's gcd-monoids \cite{Doh17I, Doh17I} where there are concrete examples of monoids that are not group-embeddable (see \cite[Proposition~4.3]{DW17}), as well as non group-embeddable Malcev type monoids from \cite{EH+}. Therefore, these classes of monoids are not covered by any previous results, including \cite{Seh22}.

\begin{thm}[Theorem~\ref{thm:lcm-monoids}]
\label{thm:C}
    Let $P$ be a cancellative right LCM monoid. Then, the C*-envelope of $\A_\lambda(P)$ is canonically *-isomorphic to the boundary quotient C*-algebra $\partial C_\lambda^*(P)$.
\end{thm}

\textbf{Acknowledgements.}
We thank Xin Li for discussions on non-abelian duality that led to simplified proofs, and for bringing the reference \cite{DW17} to our attention. We are also grateful to Cédric Arhancet for pointing out to us a small gap in the proof of Proposition~\ref{prop:auto-coaction} in an earlier version of this paper, and we thank the anonymous referees for several helpful suggestions. The third-named author acknowledges discussions held at the March 2020 FRG 248 BIRS meeting in Banff.

\section{Preliminaries}

\subsection{Coactions on operator algebras}
\label{ss:coactions}

We first discuss operator algebras and their C*-envelopes. 
The reader is referred to \cite{Arv69,Arv72,Arv08,Arv11} for the unital theory, and to \cite[\S~2.2]{DS18} for additional details on the general theory.

By an operator algebra $\A$, we mean a norm-closed subalgebra of the bounded operators $\Bz(\cH)$ on a Hilbert space $\cH$.
A \emph{C*-cover} of an operator algebra $\A$ is a pair $(\iota,\cB)$, where $\cB$ is a C*-algebra and $\iota\colon \A\to \cB$ is a completely isometric homomorphism such that $\cB=C^*(\iota(\A))$. 
The \emph{C*-envelope} of $\A$ is a C*-cover $(\kappa,C^*_\env(\A))$ of $\A$ with the \emph{co-universal property} that for every other C*-cover $(\iota,\cB)$ of $\A$, there is a surjective *-homomorphism $q_e\colon C^*(\cB) \rightarrow C^*_\env(\A)$ with $q_e \circ \iota = \kappa$ on $\A$. By this co-universal property, the C*-algebra $C^*_\env(\A)$ is unique up to canonical *-isomorphism.

Throughout this paper $\otimes$ denotes the minimal (spatial) tensor product of operator algebras, $G$ will denote a discrete group, and $u_g\in C^*(G)$ will denote the canonical unitary corresponding to $g\in G$. 
Let $\lambda\colon C^*(G)\to C_\lambda^*(G)$ denote the canonical quotient map of $C^*(G)$ onto the reduced group C*-algebra $C^*_{\lambda}(G)$. We also use $\lambda$ to denote the regular representation of $G$ in $\cU(\ell^2(G))$, so that $\lambda_g=\lambda(u_g)$, for all $g\in G$. 
We let 
\[
\Delta\colon C^*(G)\to C^*(G)\otimes C^*(G), \quad u_g\mapsto u_g\otimes u_g
\]
and 
\[
\Delta_\lambda\colon C_\lambda^*(G)\to C_\lambda^*(G)\otimes C_\lambda^*(G), \quad \lambda_g\mapsto \lambda_g\otimes \lambda_g
\]
denote the comultiplications on $C^*(G)$ and $C^*_{\lambda}(G)$, respectively, where the latter exists by Fell's absorption principle.

\begin{definition}
	\label{def:coaction}
    A \emph{coaction of a discrete group $G$ on an operator algebra $\A$} is a completely contractive homomorphism $\delta\colon \A\to\A\otimes C^*(G)$ such that
   \begin{enumerate}[\upshape(i)]
        \item  
        $(\delta \otimes \id_{C^*(G)}) \circ \delta = (\id_{\A} \otimes \Delta) \circ \delta$ (coaction identity);
       \item $\overline{\delta(\A)(I_{\A} \otimes C^*(G))} = \A \otimes C^*(G)$ (nondegeneracy).
   \end{enumerate}
   The coaction $\delta$ is \emph{normal} if the map $(\id_\A\otimes\lambda)\circ \delta\colon \A\to\A\otimes C_\lambda^*(G)$ is completely isometric.
   We write $\delta\colon G\downtouparrow \A$ to denote a coaction of $G$ on $\A$. 
\end{definition}

\begin{remark} \label{rem:auto-C*-coaction}
When $\A$ is a C*-algebra, the map $\delta$ in Definition~\ref{def:coaction} is automatically *-preserving, and is therefore a *-homomorphism. Indeed, if $\A$ is unital we do nothing, and if $\A$ is nonunital, then by \cite[\S~3]{Meyer}, $\delta$ extends to a unital completely contractive homomorphism on its unitization. Either way, a unital complete contraction is automatically positive by \cite[Proposition 2.12]{PauBook}, and therefore preserves adjoints. Thus, when $\A$ is a C*-algebra, we get the original definition of a discrete group coaction on a C*-algebra in the literature, cf. \cite[Definition~A.21]{EKQR}.
\end{remark}

Whenever we have a coaction $\delta\colon G \coacts \A$, we may define analogues of Fourier coefficients $\Ez_g\colon\A \rightarrow \A \otimes \Cz u_g$ by setting $\Ez_g(a) = (\id_{\A} \otimes \Phi_g)\circ \delta(a)$, where $\Phi_g\colon C^*(G) \rightarrow \Cz u_g$ is the standard $g$-th Fourier coefficient map on the full group C*-algebra. 
By the coaction identity, it follows that $\Ez_g(a) (I \otimes u_{g^{-1}}) \in \A \otimes \Cz I$, identified as an element $b$ in $\A$, satisfies $\delta(b) = b\otimes u_g$. We then have the following result, which shows that \cite[Definition~3.1]{DKKLL} coincides with Definition~\ref{def:coaction}.

\begin{proposition} \label{prop:auto-coaction}
Let $\A$ be an operator algebra, let $G$ be a discrete group, and suppose $\delta\colon\A \rightarrow \A \otimes C^*(G)$ is a completely contractive homomorphism. Then, $\delta$ is a coaction on $\A$ if and only if
$\sum_g\A_g^\delta$ is norm-dense in $\A$, where $\A_g^\delta$ is the spectral subspace given by $ \A_g^\delta \coloneq \{a\in \A : \delta(a)=a\otimes u_g\}$ for $g\in G$.
\end{proposition}

\begin{proof}
Suppose first that $\sum_g\A_g^\delta$ is dense in $\A$. 
It is easy to verify the coaction identity on elements of each spectral subspace, so it holds on all of $\A$. 
Next, we show that $\delta$ is nondegenerate. Clearly we have the containment $\overline{\delta(\A)(I_\cH\otimes C^*(G))} \subseteq \A\otimes C^*(G)$, so we need only prove the converse. For $g\in G$ and $a\in \A_g^\delta$, we have $\delta(a)=a\otimes u_g\in \delta(\A)$, and since $I_\cH\otimes u_{g^{-1}}\in I_\cH\otimes C^*(G)$, we have that $a\otimes I = (a \otimes u_g)(I_{\cH} \otimes u_{g^{-1}}) \in  \delta(\A)(I_\cH\otimes C^*(G))$. From density of $\sum_g\A_g^\delta$ in $\A$, we get that $\A\otimes \Cz I \subseteq \overline{\delta(\A)(I_\cH\otimes C^*(G))}$, so that $\overline{\delta(\A)(I_\cH\otimes C^*(G))} = \A \otimes C^*(G)$.

Suppose now that $\delta$ is a coaction in the sense of Definition~\ref{def:coaction}. We will show that $\sum_g \A_g$ is dense in $\A$. Let $a\in \A$ be some element, and $\epsilon >0$. Then, by nondegeneracy there are elements $a_g \in \A$, finitely many of which are nonzero, such that
\begin{equation}
    \label{eqn:estimate}
\| a \otimes I - \sum_g \delta(a_g)(I_{\cH} \otimes u_{g^{-1}}) \| < \epsilon.
\end{equation}
The formula $\Phi_g(x)=\Phi_e(xu_{g^{-1}})$ for all $x\in C^*(G)$ promotes to the following: for all $T\in \A\otimes C^*(G)$, we have 
\[
(\id_\A\otimes\Phi_g)(T)=(\id_\A\otimes\Phi_e)(T(I_H\otimes u_{g^{-1}})).
\]
This is checked first on elementary tensors, which then yields the general formula by continuity. Using this, we have 
\[
(\id_A\otimes\Phi_e)(\delta(a_g)(I_H\otimes u_{g^{-1}}))=(\id_A\otimes\Phi_g)(\delta(a_g))=\Ez_g(a_g).
\]
Since $\id_{\A} \otimes \Phi_e$ is contractive, using Equation~\eqref{eqn:estimate}, we get that
\[
 \|a \otimes I - \sum_g \Ez_g(a_g) \| < \epsilon.
\]
Since $\Ez_g(a_g)\in \A_g$, we get that $a$ can be approximated arbitrarily well by an element in $\sum_g \A_g$.
\end{proof}

\begin{remark}
    It follows from Proposition~\ref{prop:auto-coaction} that any coaction $\delta$ is automatically completely isometric. Indeed, if we let $1\colon G \rightarrow \Cz$ be the trivial representation, and we continue to write $1 \colon C^*(G) \rightarrow \Cz$ for the incuded representation on $C^*(G)$, then we have $(\id_{\A} \otimes 1) \circ \delta = \id_{\A}$, as can be verified on spectral subspaces.
\end{remark}

Let $\A$ be an operator algebra, let $G$ be a discrete group, and suppose $\delta\colon\A \rightarrow \A \otimes C^*(G)$ is a normal coaction. Several times throughout this paper we shall use the fact that, in this case, the conditional expectation $\Ez_e\colon \A\to\A$ is faithful. Indeed, we have $\Ez_e=(\id_\A\otimes\tau)(\id_\A\otimes\lambda)\circ\delta$, where $\tau$ is the canonical faithful tracial state on $C_\lambda^*(G)$. Since $\tau$ is faithful, and $(\id_\A\otimes\lambda)\circ\delta$ is completely isometric by assumption, $\Ez_e$ is faithful.
   
\begin{definition}
\label{def:reducedcoaction}
    A \emph{reduced coaction of a discrete group $G$ on an operator algebra $\A$} is a completely isometric homomorphism $\eps\colon \A\to \A\otimes C_\lambda^*(G)$ such that
   \begin{enumerate}[\upshape(i)]
        \item  
        $(\eps \otimes \id_{C^*_{\lambda}(G)}) \circ \eps = (\id_{\A} \otimes \Delta_\lambda) \circ \eps$ (coaction identity);
       \item $\overline{\eps(\A)(I_{\A} \otimes C_{\lambda}^*(G))} = \A \otimes C_{\lambda}^*(G)$ (nondegeneracy).
   \end{enumerate}
\end{definition}

\begin{remark}
As in Remark~\ref{rem:auto-C*-coaction}, when $\A$ is a C*-algebra, a reduced coaction is automatically *-preserving. Therefore, Definition~\ref{def:reducedcoaction} coincides with the notion of a reduced coaction on a C*-algebra in the literature, cf. \cite[Definition~A.72]{EKQR}.
\end{remark}

If $\delta\colon G\downtouparrow \A$ is a normal coaction, then $\delta_\lambda\coloneq(\id_\A\otimes\lambda)\circ \delta\colon \A\to\A\otimes C_\lambda^*(G)$ is a reduced coaction, as it automatically satisfies (i) and (ii) from Definition~\ref{def:reducedcoaction}. 
On the other hand, if $\eps\colon G \coacts \A$ is a reduced coaction, a similar proof as the one of Proposition \ref{prop:auto-coaction} shows that $\sum_g\A_g^\eps$ is norm-dense in $\A$, where $\A_g^\eps$ are the spectral subspaces given by
   \[
   \A_g^\eps \coloneq \{a\in \A : \eps(a)=a\otimes \lambda_g\},
   \]
for $g\in G$.
If $\delta\colon G\coacts \A$ and $\delta'\colon G\coacts \A'$ are coactions on operator algebras $\A$ and $\A'$, respectively, then a map $\phi\colon \A\to\A'$ is said to be $\delta-\delta'$-equivariant if $\delta'\circ\phi=(\phi\otimes \id_{C^*(G)})\circ\delta$. Equivariance of a map with respect to reduced coactions is defined similarly.

For the proof of Theorem~\ref{thm:main}, we need a slight strengthening of \cite[Proposition~3.4]{DKKLL} which uses Fell's absorption principle.
See \cite[\S~5.5]{CD23} and \cite{Katsoulis23} for versions of Fell's absorption principle for monoid operator algebras.
\begin{proposition}
\label{prop:5authors}
   Let $\A$ be an operator algebra, and suppose $\eps\colon G\coacts \A$ is a reduced coaction. Then, there exists a unique normal coaction $\delta\colon G\coacts \A$ such that $\delta_\lambda\coloneq(\id_\A\otimes\lambda)\circ\delta$ is equal to $\eps$. Moreover, $\A_g^\delta=\A_g^\eps$ for all $g\in G$.
   Therefore, if $\eps'\colon G\coacts\A'$ is another reduced coaction, and $\delta'\colon G\coacts \A'$ is the normal coaction satisfying $\delta'_\lambda=\eps'$, then a map $\phi\colon \A\to\A'$ is $\eps-\eps'$-equivariant if and only if it is $\delta-\delta'$-equivariant. 
\end{proposition}
\begin{proof}
Everything in the first claim except for uniqueness follows from the proof of \cite[Proposition~3.4]{DKKLL}. To prove uniqueness, suppose $\delta$ and $\bar{\delta}$ are normal coactions of $G$ on $\A$ such that $\delta_\lambda=\bar{\delta}_\lambda$. Then, for all $g\in G$ and $a\in \A_g^\eps$, we have
\[
(\id_\A\otimes\lambda)\circ \delta(a)=a\otimes\lambda_g=(\id_\A\otimes\lambda)\circ \bar{\delta}(a).
\]
By normality of $\delta$, $\id_\A\otimes\lambda$ is completely isometric on $\delta(\A)$, so that $\delta(a)=(\id_\A\otimes\lambda)\vert_{\delta(\A)}^{-1}(a\otimes\lambda_g) = a \otimes u_g$. By symmetry, $\bar{\delta}(a)=(\id_\A\otimes\lambda)\vert_{\delta(\A)}^{-1}(a\otimes\lambda_g) = a \otimes u_g$. Since $\sum_g\A_g^\eps$ is dense in $\A$, this is enough to conclude that $\delta=\bar{\delta}$.

For the second claim, observe that the lift $\delta$ of $\eps$ is explicitly given as $\delta=(\delta_\lambda\otimes\id_{C^*(G)})^{-1}\circ (\id_\A\otimes\phi)\circ \delta_\lambda$, where $\phi$ is the injective *-homomorphism
\[
\phi\colon C_\lambda^*(G)\to C_\lambda^*(G)\otimes C^*(G),\quad \lambda_g\mapsto \lambda_g\otimes u_g
\]
from Fell's absorption principle (see the proof of \cite[Proposition~3.4]{DKKLL}). 
From this, it follows that $\A_g^\eps=\A_g^\delta$ for all $g\in G$. The analogous statements hold for $\delta'$ and $\eps'$, so that $\A'^{\eps'}_g=\A'^{\delta'}_g$ for all $g\in G$. Now, if $\phi$ is $\eps-\eps'$-equivariant, then for any $g\in G$ and $a\in \A_g^\eps=\A_g^\delta$, we have $\phi(a)\in \A'^{\eps'}_g$, so that $\delta'(\phi(a))=\phi(a)\otimes u_g=\phi\otimes\id_{C^*(G)}(\delta(a))$.
The other implication is similar.
\end{proof}

\subsection{Operator algebras from categories} \label{ss:prelimcats}
In this subsection, we recall how to describe the Toeplitz C*-algebra of a left cancellative small category using the language of inverse semigroups and \'etale groupoids. 
In the present setting, these ideas originated in \cite{Spi14,Spi20}, but our treatment mostly follows \cite{LiGarI}. 
We refer the reader to \cite{Law98} for background on inverse semigroups and to \cite[\S~II]{SSW20} for background on \'etale groupoid C*-algebras (see also \cite{Renault}).

Let $\fC$ be a small category with set of objects $\fC^0$, each of which can be identified with the identity morphism on it. We denote by $\mfd, \mft\colon \fC \to \fC^0$ the domain and target maps, respectively. Given $c,d\in\fC$, the composition $cd$ is defined if $\mfd(c)=\mft(d)$. We identify $\fC$ with its set of morphisms, so that the composition of morphisms gives a partially defined binary relation on $\fC$. Note that if $\fC^0$ is a singleton, then $\fC$ is simply a monoid. Thus, we view left cancellative small categories as generalizations of monoids. A small category $\fC$ is left cancellative if for every $c,x,y\in\fC$ with $\mfd(c)=\mft(x)=\mft(y)$, the equality $cx=cy$ implies $x=y$. We will say that $\fC$ is right cancellative if the symmetric definition holds, and that $\fC$ is cancellative if it is both left and right cancellative.

When $\fC$ is a left cancellative small category, there is a left regular representation $\lambda_\fC\colon\fC\to \Bz(\ell^2(\fC))$ given by
\[
\lambda_\fC(c)e_x = 
\begin{cases}
    e_{cx} & \text{if}~\mft(x) = \mfd(c), \\
    0 & \text{otherwise},
\end{cases}
\]
for all $c,x\in \fC$, where $\{e_x : x\in\fC\}$ is the canonical orthonormal basis for $\ell^2(\fC)$. We see that for $c\in \fC$, the operator $\lambda_\fC(c)$ is a partial isometry with initial space $\ell^2(\mfd(c)\fC)$ and final space $\ell^2(c\fC)$. The following C*-algebra was defined by Spielberg \cite{Spi20}.

\begin{definition}[{\cite[Definition~11.2]{Spi20}}]
\label{def:Toeplitz}
    The \emph{reduced Toeplitz algebra of $\fC$} is the C*-algebra 
    \[
    C_\lambda^*(\fC)\coloneq C^*(\{\lambda_\fC(c) : c\in\fC\})\subseteq\Bz(\ell^2(\fC)).
    \]   
\end{definition}

The C*-algebra $C_\lambda^*(\fC)$ is called the left reduced C*-algebra of $\fC$ in \cite[Definition~2.2]{LiGarI}. 

\begin{definition}
\label{def:A(fC)}
The \emph{operator algebra of $\fC$} is 
    \[
\A_\lambda(\fC)\coloneq\overline{\alg}(\{\lambda_\fC(c) :c\in\fC\})\subseteq  C_\lambda^*(\fC).
\]
\end{definition}

The operator algebra $\A_\lambda(\fC)$ always has a contractive self-adjoint approximate identity. Indeed, the net of projections $\sum_{u\in F}\lambda_u$ as $F$ ranges over the finite subsets of $\fC^0$ is a contractive self-adjoint approximate identity of projections for the dense algebra $\alg(\{\lambda_\fC(c): c\in\fC\})$, so it is also a contractive self-adjoint approximate identity for $\A_\lambda(\fC)$. Moreover, $\A_\lambda(\fC)$ is unital whenever $\fC^0$ is finite.

A priori, we have the following description of $C_\lambda^*(\fC)$ with *-monomials:
\begin{equation}
\label{eqn:spanning1}
C_\lambda^*(\fC)=\overline{\spn}(\{\lambda_\fC(c_1)^*\lambda_\fC(d_1)\cdots \lambda_\fC(c_n)^*\lambda_\fC(d_n) : c_i,d_i\in\fC,n\geq 1\}).
\end{equation}
The language of inverse semigroups is then very convenient for understanding how *-monomials interact. 

Recall that a \emph{partial bijection of $\fC$} is a bijection between two subsets of $\fC$, i.e., a bijection $f\colon \dom(f)\to\im(f)$, where $\dom(f)$ and $\im(f)$ are subsets of $\fC$ called the \emph{domain} and $\emph{image}$ of $f$, respectively. 
The symmetric inverse monoid $\cI(\fC)$ consists of all partial bijections of $\fC$ with composition and inversion of partial bijections.
Due to left cancellation in $\fC$, each $c\in\fC$ defines a partial bijection
\begin{equation}
\label{eqn:pbi}
\mfd(c)\fC\to c\fC, \quad x\mapsto cx.
\end{equation}

Following \cite{LiGarI}, we shall use $c$ also to denote the partial bijection defined in \eqref{eqn:pbi}, so that $\dom(c)=\mfd(c)\fC$, $\im(c)=c\fC$, and $c(x)\coloneq cx$ for all $x\in \mfd(c)\fC$. The inverse $c^{-1}$ is the partial bijection of $\fC$ given by $c\fC\to \mfd(c)\fC$, $c^{-1}(cx) \coloneq x$ for all $cx\in c\fC$, though it may not make sense as an element of $\fC$. 

There is a canonical faithful representation of the inverse monoid $\cI(\fC)$ on $\ell^2(\fC)$ by partial isometries defined as follows: 
given $f\in \cI(\fC)$, there is a partial isometry $\Lambda_f\in\Bz(\ell^2(\fC))$ given by 
\begin{equation}
    \Lambda_fe_x\coloneq \begin{cases}
        e_{f(x)} & \text{if}~x\in\dom(f),\\
        0 & \text{if}~x\not\in\dom(f),
    \end{cases}
\end{equation}
for all $x\in \fC$,
and the map $f\mapsto \Lambda_f$ is injective and a homomorphism of inverse monoids in the sense that $\Lambda_{fg}=\Lambda_f\Lambda_g$ and $\Lambda_f^*=\Lambda_{f^{-1}}$ for all partial bijections $f,g\in \cI(\fC)$. In the specific case where the partial bijections come from elements of $\fC$, we have:
\begin{itemize}
    \item $\lambda_\fC(c)=\Lambda_c$ and $\lambda_\fC(c)^*=\Lambda_{c^{-1}}$ for all $c\in\fC$;
    \item $\Lambda_c^*\Lambda_d=\Lambda_{c^{-1}d}$ for all $c,d\in\fC$, where $c^{-1}d$ is the composition of $c^{-1}$ and $d$ in $\cI(\fC)$.
\end{itemize} 

This allows us to describe a general *-monomial as a partial isometry. Specifically, we have
\begin{equation}
    \label{eqn:spanning2}
\lambda_\fC(c_1)^*\lambda_\fC(d_1)\cdots \lambda_\fC(c_n)^*\lambda_\fC(d_n)=\Lambda_{c_1^{-1}d_1\cdots c_n^{-1}d_n},
\end{equation}
for all $c_i,d_i\in\fC$ and $n\geq 1$. 
The \emph{left inverse hull $I_l = I_l(\fC)$ of $\fC$} is the inverse semigroup generated by the collection $\{c : c\in\fC\}$ of partial bijections from \eqref{eqn:pbi}, and we have
\[
I_l=\{c_1^{-1}d_1\cdots c_n^{-1}d_n : c_i,d_i\in \fC, n\geq 1\}
\]
where the product $c_1^{-1}d_1\cdots c_n^{-1}d_n$ is taken inside $\cI(\fC)$. We denote by $0$ the empty function on $\fC$. Using \cite[Lemma~5.6.43]{CELY}, we see that $I_l$ contains $0$ unless $\fC$ is a monoid such that $x\fC\cap y\fC\neq\emptyset$ for all $x,y\in\fC$.  Combining \eqref{eqn:spanning1} and \eqref{eqn:spanning2}, we get that
\begin{equation}
\label{eqn:spanning3}
    C_\lambda^*(\fC)=\overline{\spn}(\{\Lambda_s : s\in I_l\}).
\end{equation}
Moreover, the map $\Lambda\colon I_l\to C_\lambda^*(\fC)$ is a faithful representation of the inverse semigroup $I_l$ by partial isometries in $C_\lambda^*(\fC)$. 
The description in \eqref{eqn:spanning3} has several important consequences. First, the composition law in $I_l$ tells us how to take products of spanning elements in $C_\lambda^*(\fC)$. Second, because the idempotents in an inverse semigroup form a semilattice (i.e., a commutative idempotent semigroup), we see that $C_\lambda^*(\fC)$ contains a canonical commutative C*-subalgebra
\[
D_\lambda(\fC)\coloneq\overline{\spn}(\{1_X : X\in\cJ\})\subseteq\ell^\infty(\fC),
\]
where $\cJ\coloneq\{\dom(s) : s\in I_l\}$ is the semilattice of constructible right ideals of $\fC$. Since the map $X\mapsto \id_X$ is a semilattice isomorphism from $\cJ$ onto the idempotent semilattice of $I_l$, we will often treat these interchangeably. We let $\cJ^\times$ denote the nonempty constructible right ideals of $\fC$.

A \emph{character} on $\cJ$ is a nonzero map $\chi\colon\cJ\to\{0,1\}$ such that $\chi(X\cap Y)=\chi(X)\chi(Y)$ for all $X,Y\in\cJ$ and $\chi(\emptyset)=0$ if $\emptyset\in\cJ$. The Gelfand spectrum of $D_\lambda(\fC)$ is canonically identified with the subspace $\Omega\subseteq\{0,1\}^\cJ$ consisting of the characters $\chi$ on $\cJ$ with the property that whenever $X,X_1,\ldots,X_n\in\cJ^\times$ satisfy $X=\bigcup_{i=1}^n X_i$, then $\chi(X)=1$ implies $\chi(X_i)=1$ for some $i=1,\ldots,n$. 
The space $\Omega$ is a locally compact totally disconnected Hausdorff space with a basis given by compact open subsets of the form 
\[
\Omega(X;\mff)\coloneq \{\chi\in \Omega : \chi(X)=1,\chi(Y)=0\text{ for all } Y\in\mff\} ,
\]
where $X\in\cJ$ and $\mff\subseteq\cJ^\times$ is a finite (possibly empty) subset such that $\bigcup_{Y\in\mff}Y\subseteq X$. For each $X\in\cJ$, we put $\Omega(X)\coloneq \Omega(X;\emptyset)$.

The inverse semigroup $I_l$ acts on $\Omega$: each $s\in I_l$ defines a partial homeomorphism
\[
\Omega(\dom(s))\to \Omega(\im(s)),\quad \chi\mapsto s.\chi,
\]
where $s.\chi(X)\coloneq \chi(s^{-1}(X\cap\im(s)))$ for all $X\in\cJ$. Let 
\[
I_l*\Omega\coloneq \{(s,\chi)\in I_l\times \Omega : \chi(\dom(s))=1\}.
\]
Define an equivalence relation on $I_l*\Omega$ by 
\[
(s,\chi)\sim (t,\chi)
\]
if there exists $X\in\cJ$ such that $\chi(X)=1$ and $s(x)=t(x)$ for all $x\in X$. We let $[s,\chi]$ denote the equivalence class of $(s,\chi)$ with respect to $\sim$. 
The transformation groupoid $I_l\ltimes\Omega$ is the quotient space $(I_l*\Omega)/\sim$ with groupoid operations determined by
\[
[s,t.\chi][t,\chi]\coloneq [st,\chi]\quad \text{ and }\quad [s,\chi]^{-1}\coloneq [s^{-1},s.\chi]
\]
for all $s,t\in I_l$ and $\chi\in\Omega$. 
The range and source maps are given by $\mfr([s,\chi])\coloneq s.\chi$ and $\mfs([s,\chi])\coloneq \chi$, respectively, for all $[s,\chi]\in I_l\ltimes \Omega$. Note that Exel calls $I_l\ltimes\Omega$ the groupoid of germs for the action $I_l\acts\Omega$, see \cite{Exel:comb}.
We tacitly identify $\Omega$ with the unit space of $I_l\ltimes\Omega$ via $\chi\mapsto [\id_X,\chi]$ for all $\chi\in \Omega$, where $X\in\cJ$ is any constructible right ideal with $\chi(X)=1$. 
Subsets of the form $[s,U]\coloneq \{[s,\chi] : \chi\in U\}$ for all $s\in I_l$ and compact open subsets $U$ of $\Omega(\dom(s))$ generate a topology on the groupoid $I_l\ltimes\Omega$
which is locally compact, \'{e}tale (both range and source maps are local homeomorphisms), and ample (there is a basis consisting of compact open bisections) though not necessarily Hausdorff. A bisection is a subset of $I_l\ltimes \Omega$ on which both the range and source maps are injective,
and any basic open subset $[s,U]$ as above is a compact open bisection.

A character $\chi$ on $\cJ$ is said to be \emph{maximal} if $\chi^{-1}(1)$ is maximal with respect to set inclusion in the collection $\{\gamma^{-1}(1) : \gamma \text{ a character of }\cJ\}$. Every maximal character lies in $\Omega$, and we denote by $\Omega_\max$ the subset of $\Omega$ consisting of maximal characters. The \emph{boundary} of $\Omega$ is the closure $\partial\Omega\coloneq\overline{\Omega_\max}$, which is a closed and invariant subset of $\Omega$. 
Given $X\in\cJ$ and a finite subset $\mff\subseteq\cJ^\times$ with $\bigcup_{Y\in\mff}Y\subseteq X$, we put $\partial\Omega(X;\mff)\coloneq\partial\Omega\cap\Omega(X;\mff)$ and $\partial\Omega(X)\coloneq\partial\Omega(X;\emptyset)$.
Compact open subsets of this form are a basis for the topology on $\partial\Omega$. 

\begin{remark}
\label{rmk:bdry}
If $0\notin I_l$, then $\partial\Omega= \Omega_{\max} = \{\chi_\infty\}$ is a single point, where $\chi_\infty\colon \cJ\to\{0,1\}$ is the unique maximal character defined by $\chi(X)=1$ for all $X\in\cJ$. 
\end{remark}

Let us recall some terminology from \cite[Definition 11.5]{Exel:comb}. A subset $F\subseteq\cJ$ of constructible right ideals is said to be a \emph{cover} for $X\in\cJ$ if $Z\subseteq X$ for all $Z\in F$ and for every $Y\in\cJ$ with $Y\subseteq X$, there exists $Z\in F$ such that $Z\cap Y\neq\emptyset$. 
If $0\in I_l$, then $\partial\Omega$ is precisely the set of tight characters of $\cJ$ in the sense of Exel, see \cite[Theorem~12.9]{Exel:comb}. Precisely, this means that $\chi\in\Omega$ lies in $\partial\Omega$ if and only if whenever $X\in \cJ^\times$ and $F$ is a finite cover for $X$, we have $\chi(Y)=1$ for some $Y\in F$. This implies that $\partial\Omega(X,\mff)=\emptyset$ whenever $\mff$ is a cover for $X$ (cf. \cite{DM14} and \cite{Exel21}). 

 Next, we explain how each character $\chi \in \Omega$ gives rise to an analogue of a left regular representation on those groupoid elements whose source equals $\chi$.
 This makes sense for any \'etale groupoid. Since we shall not need to consider such representations for non-Hausdorff groupoids, we shall assume that $I_l\ltimes\Omega$ is Hausdorff.
 Note that $(I_l\ltimes\Omega)_\chi\coloneq\{[s,\chi]\in I_l\ltimes\Omega : s\in I_l, \chi(\dom(s))=1\}$ is discrete since the groupoid is \'etale, and define the left regular representation of $C_c(I_l\ltimes\Omega)$ associated with $\chi$ as
\[
\rho_\chi\colon C_c(I_l\ltimes\Omega)\to \Bz(\ell^2((I_l\ltimes\Omega)_\chi)),
\]
given by 
\[
\rho_\chi(f)\delta_{[t,\chi]}=\sum_{[s,t.\chi]\in (I_l\ltimes\Omega)_{t.\chi}}f([s,t.\chi])\delta_{[st,\chi]},
\]
for all $f\in C_c(I_l\ltimes\Omega)$ and $[t,\chi]\in (I_l\ltimes\Omega)_\chi$. Similarly, given $\chi\in\partial\Omega$, we let 
\[
\partial\rho_\chi\colon C_c(I_l\ltimes\partial\Omega)\to \Bz(\ell^2((I_l\ltimes\partial\Omega)_\chi))
\]
denote the left regular representation of $C_c(I_l\ltimes\partial\Omega)$ associated with $\chi$. 
We then define $C_r^*(I_l\ltimes\Omega)$ as the completion of $\bigoplus_{\chi\in\Omega}\rho_\chi(C_c(I_l\ltimes\Omega))$ in $\Bz(\bigoplus_{\chi\in\Omega}\ell^2((I_l\ltimes\Omega)_\chi)))$, 
and, similarly, $C_r^*(I_l\ltimes\partial\Omega)$ is the completion of $\bigoplus_{\chi\in \partial\Omega}\rho_\chi(C_c(I_l\ltimes \partial\Omega))$ in $\Bz(\bigoplus_{\chi\in \partial \Omega}\ell^2((I_l\ltimes \partial\Omega)_\chi)))$. 
Both $\rho_{\chi}$ and $\partial \rho_{\chi}$ extend to *-homomorphisms (still denoted $\rho_{\chi}$ and $\partial \rho_{\chi}$) of $C_r^*(I_l\ltimes \Omega)$ and $C_r^*(I_l\ltimes\partial\Omega)$, respectively.
From the analysis in \cite[\S~3]{LiGarI} (in particular, \cite[Corollary~3.4]{LiGarI}) and following \cite[Proposition~11.4]{Spi20}, we see that if the groupoid $I_l\ltimes\Omega$ is Hausdorff, then there is a *-isomorphism
\begin{equation}
\label{eqn:Jackmap}
\mfj\colon C_r^*(I_l\ltimes\Omega)\to C_\lambda^*(\fC), \quad 1_{[s,\Omega(\dom(s)))]}\mapsto \Lambda_s,
\end{equation}   
for all $s\in I_l$.

By \cite[Lemma~4.1(i)]{LiGarI}, we know that $I_l\ltimes\Omega$ is Hausdorff if and only if for all $s\in I_l$, there exists a finite (possibly empty) set $F\subseteq\cJ^\times$ such that $\{x\in\dom(s) :s(x)=x\}=\bigcup_{X\in F}X$. For instance, this means that $I_l\ltimes\Omega$ is Hausdorff whenever $\fC$ is cancellative and finitely aligned in the sense of \cite[Definition~3.2]{Spi20}. Thus, in order to use the identification of C*-algebras in \eqref{eqn:Jackmap}, we assume $I_l \ltimes \Omega$ is Hausdorff.

\begin{definition}
\label{def:bdry}
Let $\fC$ be a cancellative small category, and suppose that $I_l\ltimes\Omega$ is Hausdorff. The \emph{(reduced) boundary quotient of $C_\lambda^*(\fC)$} is the C*-algebra $\partial C_\lambda^*(\fC)\coloneq C_r^*(I_l\ltimes\partial\Omega)$.
\end{definition}
The \emph{boundary quotient map} is the surjective *-homomorphism
\begin{equation}
\label{eqn:bdryqmap}
    q_\partial\colon C_r^*(I_l\ltimes\Omega)\to C_r^*(I_l\ltimes\partial\Omega)
\end{equation}
determined by $q_\partial(f)=f\vert_{I_l\ltimes\partial\Omega}$ for all $f\in C_c(I_l\ltimes\Omega)$. Note that such a *-homomorphism exists because the restriction map $C_c(I_l\ltimes\Omega)\to C_c(I_l\ltimes\partial\Omega)$ is contractive for the reduced C*-norms: the norm on $C_c(I_l\ltimes\Omega)$ is defined by a supremum over $\Omega$, where as the norm on $C_c(I_l\ltimes\partial\Omega)$ is defined by a supremum over $\partial\Omega$, a subset of $\Omega$.

When $I_l\ltimes\Omega$ is Hausdorff, the isomorphism in \eqref{eqn:Jackmap} justifies Definition~\ref{def:bdry}. In the non-Hausdorff setting, there is another candidate for a groupoid model for $C_\lambda^*(\fC)$ (see \cite[\S~3]{LiGarI}), and these two groupoids can be different, see \cite{Sch}.


\section{Existence of coactions on C*-envelopes} \label{s:exists}

In this section, we extend Katayama duality for normal coactions of discrete groups on C*-algebras \cite[Theorem~8]{Kat} to normal coactions of discrete groups on operator algebras. The proof is a straightforward generalization of the C*-algebra version.
We then use this Katayama duality to prove our main result (Theorem \ref{thm:main}) that a normal coaction on an operator algebra extends to a coaction on the C*-envelope.

Let $\A\subseteq\Bz(\cH)$ be an operator algebra, and suppose $\delta\colon G\coacts\A$ is a coaction by a discrete group $G$. 
Let $M\colon c_0(G)\to \Bz(\ell^2(G))$ be the canonical representation by diagonal multiplication operators, and define $j_{c_0(G)}\colon c_0(G)\to \Bz(\cH\otimes\ell^2(G))$ by $j_{c_0(G)}(f)\coloneq I\otimes M_f$ for all $f\in c_0(G)$. We view $\delta_\lambda\coloneq (\id_\A\otimes\lambda)\circ\delta$ as a homomorphism $\A\to \Bz(\cH\otimes\ell^2(G))$. 

\begin{definition}
\label{def:coactioncproduct}
 The (reduced) \emph{crossed product} of $\A$ by the coaction $\delta$ of $G$ is the operator algebra 
\begin{equation}
\A\rtimes_\delta G\coloneq\overline{\alg}(\{\delta_\lambda(a)j_{c_0(G)}(f) : a\in\A, f\in c_0(G)\})\subseteq\Bz(\cH\otimes\ell^2(G)).
\end{equation}
\end{definition}

For $g\in G$, there is a completely isometric automorphism $\hat{\delta}_g\colon \A\rtimes_\delta G\to\A\rtimes_\delta G$ such that $\hat{\delta}_g(x)=(I\otimes\rho_g^*)x(I\otimes\rho_g)$ for all $x\in \A\rtimes_\delta G$, where $\rho\colon G\to \cU(\ell^2(G))$ is the right regular representation of $G$. We call $\hat{\delta}$ the \emph{dual action} of $G$ on $\A\rtimes_\delta G$. On generators, we have 
\[
\hat{\delta}_g(\delta_\lambda(a)j_{c_0(G)}(f))=\delta_\lambda(a)j_{c_0(G)}(\sigma_g(f)),
\]
for all $a\in \A$ and $f\in c_0(G)$, where $\sigma_g\colon c_0(G)\to c_0(G)$ is given by $\sigma_g(f)(h)=f(hg)$ for all $g,h\in G$ and $f\in c_0(G)$. 

Consider the maps
\begin{itemize}
    \item $k_\A(a)\coloneq \delta_\lambda(a)\otimes I_{\ell^2(G)}$ for all $a\in \A$;
    \item $k_{c_0(G)}(f)\coloneq I\otimes ((M\otimes M)\circ\nu(f))$ for all $f\in c_0(G)$;
    \item and $k_G(g)\coloneq I\otimes I\otimes \lambda_g$ for all $g\in G$,
\end{itemize}
where $\nu\colon c_0(G)\to \cM(c_0(G)\otimes c_0(G))$ is given by $\nu(f)(g,h)\coloneq f(gh^{-1})$ for all $g,h\in G$ and $f\in c_0(G)$. We define the \emph{reduced double crossed product} $\A\rtimes_\delta G\rtimes_{\hat{\delta}}^rG$ to be the operator algebra 
\[
\A\rtimes_\delta G\rtimes_{\hat{\delta}}^rG\coloneq \overline{\alg}(\{k_\A(\A)k_{c_0(G)}(c_0(G))k_G(G)\})\subseteq\Bz(\cH\otimes\ell^2(G)\otimes\ell^2(G)).
\]
\begin{remark}
The operator algebra $\A\rtimes_\delta G\rtimes_{\hat{\delta}}^rG$ coincides with the relative crossed product of $\A\rtimes_\delta G$ by the action $\hat{\delta}$ of $G$ with respect to the $\hat{\delta}$-admissible cover $C^*(\A\rtimes_\delta G)\subseteq \Bz(\cH\otimes\ell^2(G))$ of $\A\rtimes_\delta G$, as defined in \cite[Definition~3.2]{KR}. 
\end{remark}

Next, define unitary operators $U,S\in \cU(\ell^2(G)\otimes\ell^2(G))$ by setting 
\[
Ue_g\otimes e_h\coloneq e_g\otimes e_{gh}\quad\text{ and } \quad Se_g\otimes e_h\coloneq e_g\otimes e_{h^{-1}}
\]
for all $g,h\in G$, where $\{e_g:g\in G\}$ is the canonical orthonormal basis for $\ell^2(G)$. 
There is a reduced coaction $\hat{\hat{\delta}}$ of $G$ on $\A\rtimes_\delta G\rtimes_{\hat{\delta}}^rG$ given by
\begin{equation}    
\label{eqn:deltahathat}
\hat{\hat{\delta}}(x)\coloneq(I_\cH\otimes I_{\ell^2(G)}\otimes U)(x\otimes I_{\ell^2(G)})(I_\cH\otimes I_{\ell^2(G)}\otimes U^*)
\end{equation}
for all $x\in \A\rtimes_\delta G\rtimes_{\hat{\delta}}^rG$. For $a\in \A$, $f\in c_0(G)$, and $g\in G$, a straightforward computation yields the formula 
\[
\hat{\hat{\delta}}(k_\A(a) k_{c_0(G)}(f) k_G(g)) = k_\A(a) k_{c_0(G)}(f) k_G(g) \otimes \lambda_g.
\]
We are now ready to extend Katayama duality \cite[Theorem~8]{Kat} to the non-selfadjoint setting. Recall that when $\delta$ is normal, $\delta_\lambda$ is a completely isometric isomorphism from $\A$ onto $\delta_\lambda(\A)$.

\begin{theorem}[Katayama duality for operator algebras]
\label{thm:Katduality}
Let $\A\subseteq\Bz(\cH)$ be an operator algebra, let $G$ be a discrete group, and suppose $\delta\colon G\downtouparrow \A$ is a normal coaction. Then, there is a completely isometric isomorphism 
\[
\Psi \colon \A\rtimes_\delta G\rtimes^r_{\hat{\delta}}G \to \delta_\lambda(\A)\otimes \Kz
\]
such that the reduced coaction $\hat{\hat{\delta}}$ from \eqref{eqn:deltahathat} on the double crossed product is conjugated to the reduced coaction $\tilde{\delta}$ of $G$ on $\delta_\lambda(\A)\otimes \Kz$
given by
\begin{equation}
    \label{eqn:tildedelta}
\tilde{\delta}(x)=(I_\cH\otimes I_{\ell^2(G)}\otimes U^*)[(\delta_\lambda\otimes\Sigma)\circ (\delta_\lambda\circ (\delta_\lambda^{-1}\otimes\id))(x)](I_\cH\otimes I_{\ell^2(G)}\otimes U),
\end{equation}
for all $x\in \delta_\lambda(\A)\otimes \Kz$, where $\Sigma\colon C_\lambda^*(G)\otimes\Kz\to \Kz\otimes C_\lambda^*(G)$ is the flip map.
\end{theorem}
\begin{proof}
The proof is essentially the same as the proof of Katayama duality for C*-algebras as given in \cite[Lemma~A.70~and~Theorem~A.69]{EKQR}, so we give only a sketch. Consider the unitary $V\coloneq I_\cH\otimes US\in \cU(\cH\otimes\ell^2(G)\otimes\ell^2(G))$. Direct calculations yield the following formulas:
\begin{enumerate}[\upshape(i)]
  \item $\Ad(V)(k_\A(a))=\delta_\lambda(a)\otimes \lambda_g$ for $g\in G$ and $a\in \A_g$;
\item $\Ad(V)(k_{c_0(G)}(f))=I\otimes I\otimes M_f$ for $f\in c_0(G)$;
    \item $\Ad(V)(k_G(g))=I\otimes I\otimes\rho_g$ for $g\in G$.
\end{enumerate}
Define $\Psi$ to be the restriction to the double crossed product of conjugation by the unitary $V$.
It follows from (i), (ii), and (iii) that $\Psi$ carries $\A\rtimes_\delta G\rtimes_{\hat{\delta}}^rG$ completely isometrically into $\delta_\lambda(\A)\otimes\Kz$.

Next, we show that $\Psi$ is surjective.
Observe from (i), (ii), and (iii) again and the fact that $\overline{c_0(G) C_\rho^*(G)}=\Kz$, we have 
\begin{align*}
 \Ad(V)(k_{c_0(G)}(c_0(G)))\Ad(V)(k_G(G))&=\Cz I_\cH\otimes \Cz I_{\ell^2(G)}\otimes \overline{c_0(G)  C_\rho^*(G)}\\
 &=\Cz I_\cH\otimes \Cz I_{\ell^2(G)}\otimes\Kz,
\end{align*}
so it suffices to show that 
\[
\overline{\Ad(V)(k_\A(\A))[\Cz I_\cH\otimes \Cz I_{\ell^2(G)}\otimes\Kz}]= \delta_\lambda(\A)\otimes \Kz.
\]

Since $\delta_\lambda$ is nondegenerate and $\Kz = \overline{C^*_{\lambda}(G) c_0(G)}$, we see that
\begin{align*}
\overline{\delta_\lambda(\A)(\Cz I_\cH\otimes\Kz)}&=\overline{\delta_\lambda(\A)(\Cz I_\cH\otimes C_\lambda^*(G))(\Cz I_\cH\otimes c_0(G))}\\
&=\overline{(\A\otimes C_\lambda^*(G))(\Cz I_\cH\otimes c_0(G))}\\
&=\A\otimes\Kz.
\end{align*}
A straightforward computation then shows that 
\begin{align*}
    \Ad(V)(k_\A(\A))[\Cz I_\cH\otimes & \Cz I_{\ell^2(G)}\otimes\Kz]
        =(\delta_\lambda\otimes\id_{\Kz})[\delta_\lambda(\A)(\Cz I_\cH\otimes\Kz)],
\end{align*}
so by taking closures and using what we established above, we obtain 
\begin{align*}
\overline{\Ad(V)(k_\A(\A))[\Cz I_\cH\otimes \Cz I_{\ell^2(G)}\otimes\Kz}]&= \overline{(\delta_\lambda\otimes\id_\Kz)[\delta_\lambda(\A)(\Cz I_\cH\otimes\Kz)]}\\
&=(\delta_\lambda\otimes\id_\Kz)[\overline{\delta_\lambda(\A)(\Cz I_{\cH}\otimes\Kz)}]\\
&=\delta_\lambda(\A)\otimes\Kz.
\end{align*}
This proves that $\Psi$ is surjective, so $\Psi$ is a completely isometric isomorphism.

Finally, the double crossed product is invariant for the induced double-dual coaction $\hat{\hat{\delta}}$,
so the fact that $\Psi$ conjugates $\hat{\hat{\delta}}$ into $\tilde{\delta}$ given in formula \eqref{eqn:tildedelta} follows from a direct calculation.
\end{proof}

Theorem~\ref{thm:Katduality} extends Katayama's original result \cite[Theorem~8]{Kat} by Remark~\ref{rem:auto-C*-coaction} and Proposition~\ref{prop:auto-coaction}. Before turning to our main result, we need a technical lemma.
\begin{lemma}
\label{lem:coaction_cai}
   Suppose $\A$ has a contractive self-adjoint approximate identity, and assume that $\delta\colon G\coacts\A$ is a coaction by a discrete group $G$. Then, the coaction crossed product $\A\rtimes_\delta G$ has a contractive self-adjoint approximate identity.
\end{lemma}
\begin{proof} 
First, suppose we have a contractive self-adjoint approximate identity $\{a_{\alpha}\}_{\alpha \in I}$ of $\A$ and let $\Ez_e$ be the canonical conditional expectation onto $\A_e$, as described in \S~\ref{ss:coactions}. 
We will show that $\{\Ez_e(a_{\alpha})\}_{\alpha \in I}$, which is a net in $\A_e$, is a contractive self-adjoint approximate identity for $\A$. It is clearly contractive and self-adjoint, because $\Ez_e$ is a contractive idempotent. Now, for $a\in \A_g$ and $g\in G$ we first show that $a\Ez_e(a_{\alpha}) \rightarrow a$. Indeed,
\[
a = \lim_{\alpha}\Ez_g(aa_{\alpha}) = \lim_{\alpha}\Ez_g(a\Ez_e(a_{\alpha})) = 
\Ez_g(a) \lim_{\alpha}\Ez_e(a_{\alpha}) = a \lim_{\alpha}\Ez_e(a_{\alpha}).
\]
Now let $a\in \A$ and $\epsilon>0$.
As the linear span of spectral subspaces is dense, there is a finite subset $F\subseteq G$ and elements $a_g \in \A_g$ for $g\in F$ such that $\sum_{g\in F} a_g$ is $\epsilon$-close to $a$. Hence, since $\{\Ez_e(a_{\alpha})\}_{\alpha \in I}$ is an approximate identity for $\sum_{g\in F} a_g$ for arbitrary $\epsilon>0$, it is also an approximate identity for $a$.

Hence, we may assume without loss of generality that $\{a_{\alpha}\}_{\alpha \in I}$ is a contractive self-adjoint approximate unit for $\A$ which is in $\A_e$. Let $\{\chi_F\}_F$ be the net of characteristic functions supported on finite subsets $F\subseteq G$, which we denote by $\Fin(G)$. We will show that the net of real parts $\{\frac{\chi_F a_{\alpha} + a_{\alpha} \chi_F}{2} \}_{(\alpha,F) \in I \times \Fin(G)}$, which is comprised of contractive self-adjoint elements, is an approximate identity for $\A\rtimes_\delta G$.

It suffices to show that $\{a_{\alpha} \chi_F \}_{(\alpha,F) \in I \times \Fin(G)}$ is both a left and a right approximate identity on generators of the form $\delta_\lambda(a)j_{c_0(G)}(\chi_E)$, where $a\in \A_g$ for $g\in G$, and $E\subseteq G$ is finite, and the proof for $\{\chi_F a_{\alpha}\}$ is then handled similarly. 
So, for $\{a_{\alpha} \chi_F \}_{(\alpha,F) \in I \times \Fin(G)}$, if we take $F$ finite such that $gE, E\subseteq F$, and by multiplying on the right we see that
\begin{align*}
\delta_\lambda(a)j_{c_0(G)}(\chi_E) \delta_\lambda(a_{\alpha})j_{c_0(G)}(\chi_F) &= 
\delta_\lambda(a)(I \otimes M_{\chi_E})  (a_{\alpha} \otimes I) (I\otimes M_{\chi_F})\\  
&= \delta_\lambda(aa_{\alpha})j_{c_0(G)}(\chi_E)
\end{align*}
converges to $\delta_\lambda(a)j_{c_0(G)}(\chi_E)$.
By multiplying on the left,
\begin{align*}
\delta_\lambda(a_{\alpha})j_{c_0(G)}(\chi_F) \delta_\lambda(a)j_{c_0(G)}(\chi_E) &= (a_{\alpha} \otimes I) (I\otimes M_{\chi_F}) (a \otimes \lambda_g)(I \otimes M_{\chi_E}) \\
&= (a_{\alpha}\otimes I) (I \otimes M_{\chi_F})(I\otimes M_{\chi_{gE}}) (a\otimes \lambda_g) \\
&= (a_{\alpha}\otimes I)(I\otimes M_{\chi_{gE}}) (a\otimes \lambda_g) \\
&= \delta_\lambda(a_{\alpha}a) j_{c_0(G)}(\chi_E)  
\end{align*}
converges to $\delta_\lambda(a) j_{c_0(G)}(\chi_E)$. Hence, we see that $\{a_{\alpha} \chi_F \}_{(\alpha,F) \in I \times \Fin(G)}$ is an approximate identity for linear generators of $\A\rtimes_\delta G$, and therefore for their closure as well.
\end{proof}

We are now ready for the main result of this paper.

\begin{theorem}
\label{thm:main}
    Let $\A$ be an operator algebra with a contractive self-adjoint approximate identity, and let $\kappa_\A\colon \A \to C^*_\env(\A)$ be the canonical completely isometric inclusion.
    Suppose $\delta\colon G \downtouparrow\A$ is a normal coaction of a discrete group $G$. 
    Then, there exists a normal coaction $\delta_\env\colon G \downtouparrow C_\env^*(\A)$ such that $\delta_\env\circ\kappa_\A=(\kappa_\A\otimes\id_{C^*(G)})\circ\delta$.
\end{theorem}

\begin{proof}
Let $\delta_\lambda = (\id \otimes \lambda)\circ \delta$ be the reduced coaction of $G$ on $\A$ associated with $\delta$.
By Proposition~\ref{prop:5authors}, it suffices to find a reduced coaction $\delta_{\env,r}$ on the C*-envelope $C^*_\env(\A)$ such that $\kappa_\A$ is $\delta_\lambda - \delta_{\env,r}$-equivariant. 

Let $\Psi\colon \A\rtimes_\delta G\rtimes^r_{\hat{\delta}}G\to \delta_\lambda(\A)\otimes \Kz$ be the Katayama isomorphism from Theorem~\ref{thm:Katduality},
and let $\Psi^{-1}_\env\colon C_\env^*(\delta_\lambda(\A)\otimes \Kz)\to C_\env^*(\A\rtimes_\delta G\rtimes^r_{\hat{\delta}}G)$ be the *-isomorphism induced from $\Psi^{-1}$ between the C*-envelopes.
By the co-universal property of the C*-envelope, the dual action $\hat{\delta}$ of $G$ on $\A\rtimes_\delta G$ extends to an action, also denoted $\hat{\delta}$, of $G$ on $C_\env^*(\A\rtimes_\delta G)$ (see \cite[Lemma~3.4]{KR}). By Lemma~\ref{lem:coaction_cai}, $\A\rtimes_\delta G$ has a contractive self-adjoint approximate identity, so \cite[Theorem~2.5]{Katsoulis} gives us a *-isomorphism $\theta \colon C^*_\env(\A\rtimes_\delta G \rtimes^r_{\hat{\delta}} G) \to  C^*_\env(\A\rtimes_\delta G) \rtimes^r_{\hat{\delta}} G$ that maps generators to generators. 
We summarize the above discussion in the following commutative diagram.
\begin{equation}
    \label{di:main}
\begin{tikzcd}
    C_\env^*(\delta_\lambda(\A)\otimes\Kz) \arrow[r,"\Psi^{-1}_\env","\cong"'] &  C_\env^*(\A\rtimes_\delta G\rtimes_{\hat{\delta}}^rG)\arrow[r, "\theta","\cong"'] & C_\env^*(\A\rtimes_\delta G)\rtimes_{\hat{\delta}}^r G\\  \delta_\lambda(\A)\otimes\Kz\arrow[u,hook,"\kappa_{\delta_\lambda(\A)\otimes\Kz}"] \arrow[r,"\Psi^{-1}","\cong"'] &\A\rtimes_\delta G\rtimes_{\hat{\delta}}^rG \nospacepunct{.}\arrow[u,hook,"\hat{\hat{\kappa}}"] \arrow[ur,hook,"\varphi"]&
\end{tikzcd}
\end{equation}
Here, $\kappa_{\delta_\lambda(\A)\otimes\Kz}$ and $\hat{\hat{\kappa}}$ are the canonical completely isometric inclusions into the C*-envelopes,
and $\varphi \coloneq \theta \circ \hat{\hat{\kappa}}$.

Let $\hat{\hat{\delta}}\colon G \downtouparrow \A \rtimes_\delta G \rtimes^r_{\hat{\delta}} G$ denote the canonical reduced coaction of $G$ on the double crossed product, and let $\eps\colon G \downtouparrow C_\env^*(\A\rtimes_\delta G)\rtimes^r_{\hat{\delta}}G$ denote the canonical reduced coaction on the C*-crossed product. 
As $\varphi$ maps generators to generators, it is $\hat{\hat{\delta}} - \eps$-equivariant.
Let $\tilde{\delta}$ denote the reduced coaction of $G$ on $\delta_\lambda(\A) \otimes \Kz$ from \eqref{eqn:tildedelta} in Theorem~\ref{thm:Katduality}. 
Define $\tilde\delta_\env\colon G\coacts C^*_{\env}(\delta_\lambda(\A)\otimes \Kz)$ as the reduced coaction given by $\tilde\delta_\env\coloneq (\Psi_\env\circ\theta^{-1}\otimes\id_{C_\lambda^*(G)})\circ\eps\circ \theta \circ (\Psi^{-1})_\env$. Here, we used that $(\Psi^{-1}_\env)^{-1}=\Psi_\env$. 
By commutativity of the diagram \eqref{di:main} and equivariance of $\varphi$ and $\Psi$, we have 
\begin{align*}
\eps\circ \theta \circ \Psi^{-1}_\env\circ \kappa_{\delta_\lambda(\A)\otimes\Kz} 
=([\theta\circ (\Psi^{-1})_{\env}\circ\kappa_{\delta_\lambda(\A)\otimes\Kz}]\otimes\id_{C_\lambda^*(G)})\circ\tilde\delta,
\end{align*}
so $\tilde\delta_\env\circ\kappa_{\delta_\lambda(\A)\otimes\Kz}=(\kappa_{\delta_\lambda(\A)\otimes\Kz}\otimes\id_{C_\lambda^*(G)})\circ\tilde\delta$, i.e.,  $\kappa_{\delta_{\lambda}(\A) \otimes \Kz}$ is $\tilde\delta - \tilde\delta_\env$-equivariant.

Next, let $P_e$ denote the rank-one projection in $\Kz$ onto the subspace spanned by the point-mass function at the identity element $e$ of $G$. We now show that $\delta_\lambda(\A)\otimes P_e$ is invariant under $\tilde\delta$. For this, it suffices to check that $\tilde\delta(a \otimes P_e)\in \delta_\lambda(\A)\otimes P_e\otimes C^*_\lambda(G)$ for all $a \in \A_g$ and $g\in G$ (recall the explicit definition of $\tilde{\delta}$ from the statement of Theorem~\ref{thm:Katduality}).
For all $a \in \A_g$ and $g\in G$, we have
\begin{align*}
    \tilde{\delta}(\delta_\lambda(a)\otimes P_e)
    &=(I_\cH\otimes I_{\ell^2(G)}\otimes U^*)[\delta_\lambda(a)\otimes P_e \otimes \lambda_g](I_\cH\otimes I_{\ell^2(G)}\otimes U)\\
    &=\delta_\lambda(a)\otimes U^*(P_e\otimes \lambda_g)U\\
    &=\delta_\lambda(a)\otimes P_e\otimes \lambda_g.
\end{align*}
Let $\tilde\delta'$ denote the restriction of $\tilde\delta$ to $\delta_\lambda(\A)\otimes P_e$. 

By \cite[Corollary~2.7]{DEG}, there is a *-isomorphism $C^*_\env(\delta_\lambda(\A) \otimes \Kz) \to C^*_\env(\delta_\lambda(\A)) \otimes \Kz$ given by $\kappa_{\delta_\lambda(\A)}(\delta_\lambda(a))\otimes K \mapsto \kappa_{\delta_\lambda(\A)\otimes\Kz}(\delta_\lambda(a)\otimes K)$ for all $a\in\A$ and $K\in\Kz$.
A similar argument as above now shows that the C*-subalgebra $C^*_\env(\delta_\lambda(\A)) \otimes P_e$ of $C^*_\env(\delta_\lambda(\A) ) \otimes \Kz \cong C^*_\env(\delta_\lambda(\A) \otimes \Kz)$ is $\tilde\delta_{\env}$-invariant.
Let $\tilde\delta_\env'$ denote the restriction of $\tilde\delta_\env$ on $C^*_\env(\delta_\lambda(\A))\otimes P_e$.
It now follows that the inclusion $\kappa_{\delta_\lambda(\A)\otimes \id}$ is $\tilde{\delta}' - \tilde{\delta}_\env'$-equivariant.
This is summarized in the right-most part of the diagram below.
\begin{equation*}
\begin{tikzcd}[sep=small]
    C_\env^*(\A) \arrow[r,"\alpha_\env"] &  C_\env^*(\delta_\lambda(\A)) \otimes P_e \arrow[r,hook,"\textrm{incl}"] & C_\env^*(\delta_\lambda(\A))\otimes \Kz \arrow[r,"\cong"] &  C_\env^*(\delta_\lambda(\A)\otimes \Kz) \\
    \A \arrow[r,"\alpha"] \arrow[u, "\kappa_\A"] & \delta_\lambda(A)\otimes P_e \arrow[u,"\kappa_{\A\otimes \id}"] \arrow[rr,hook,"\textrm{incl}"] & & \delta_\lambda(A)\otimes \Kz \arrow[u,"\kappa_{\delta_{\lambda(\A)}\otimes \Kz}"]\nospacepunct{.}
\end{tikzcd}
\end{equation*}
The map $\alpha\colon \A \to \delta_\lambda(\A) \otimes P_e$ given by $\alpha(a) = \delta_\lambda(a)\otimes P_e$ for all $a\in \A$ is a completely isometric isomorphism. 
For $g\in G$ and $a\in A_g$, we have 
\[
\tilde{\delta}'\circ\alpha(a)=\tilde{\delta}'(\delta_\lambda(a)\otimes P_e)=\delta_\lambda(a)\otimes P_e\otimes \lambda_g=(\alpha\otimes\id)(a\otimes\lambda_g)=(\alpha\otimes\id)\delta_\lambda(a)
\]
which shows that $\alpha$ is $\delta_{\lambda}-\tilde\delta'$-equivariant. 

Finally, let $\alpha_\env\colon C^*_\env(\A) \to C^*_\env(\delta_\lambda(\A))\otimes P_e \cong C^*_\env(\delta_\lambda(\A) \otimes P_e)$ be the composition of the canonical *-isomorphism $C^*_\env(\A) \cong C^*_\env(\delta_\lambda(\A))\otimes P_e$ and the *-isomorphism between the C*-envelopes induced from $\alpha$, and consider the reduced coaction $\delta_{\env,r}\coloneq(\alpha_{\env}^{-1} \otimes \id) \circ \tilde{\delta}' \circ \alpha_{\env}$ of $G$ on $C^*_\env(\A)$. 
Using commutativity of the above diagram and that $\kappa_{\delta_\lambda(\A)\otimes \id}$ and $\alpha$ are equivariant, it follows that
\begin{align*}
    \delta_{\env,r}\circ\kappa_\A =(\alpha_{\env}^{-1} \otimes \id) \circ \tilde{\delta}' \circ\kappa_{\delta_\lambda(\A)\otimes \id}\circ\alpha = \\
    (\kappa_{\A}\circ\alpha^{-1}\otimes \id) \circ (\alpha\otimes\id)\circ \delta_\lambda 
    =(\kappa_{\A}\otimes\id)\circ\delta_\lambda,
\end{align*}
i.e., $\kappa_\A$ is $\delta_\lambda-\delta_{\env,r}$-equivariant. 
This means that $C^*_\env(\A)$ admits a reduced coaction $\delta_{\env,r}$ which extends the reduced coaction $\delta_\lambda$ on $\A$.
\end{proof}

Below, $C_\env^*(\A,\delta)$ denotes the equivariant C*-envelope for the coaction $\delta\colon G\coacts\A$, as defined in \cite[Definition~3.7]{DKKLL}. As a corollary to Theorem~\ref{thm:main}, we answer the question posed by Sehnem in \cite[Remark~5.2]{Seh22} for all normal coactions, showing that the notion of equivariant C*-envelope of a normal coaction is superfluous.

\begin{corollary}
\label{cor:main}
    Let $\A$ be an operator algebra with a contractive self-adjoint approximate identity, and suppose $\delta\colon G \downtouparrow\A$ is a normal coaction of a discrete group $G$.
    Then, the canonical map $C_\env^*(\A,\delta)\to C_\env^*(\A)$ is a *-isomorphism.
\end{corollary}

\begin{remark} \label{r:sehnems}
Although we do not discuss product systems of C*-correspondences over group-embeddable monoids in this paper, it is worthwhile to sketch the simplified proof of \cite[Theorem 5.1]{Seh22} that we may now obtain. More precisely, we can now dispense with the context-specific techniques used in \cite[Lemma 3.7]{Seh22} and \cite[Section 4]{Seh22} by appealing to Theorem \ref{thm:main} and Corollary \ref{cor:main} combined with \cite[Corollary 3.5]{Seh22}. 

To provide a sketch of the proof, in what follows we adhere to terminology and conventions from \cite{Seh22}. Let $\mathcal{E}$ be a product system of C*-correspondences over a monoid $P$ embeddable in a group $G$.  To show that the C*-envelope of the tensor algebra $\mathcal{T}_{\lambda}^+(\mathcal{E})$ is the reduced Fell bundle C*-algebra completion of $([A \times_{\mathcal{E}} P]_g)_{g\in G}$ (which is denoted by $C^*_r(([A \times_{\mathcal{E}} P]_g)_{g\in G})$) we first note that by Theorem \ref{thm:main} we have that $C^*_\env(\mathcal{T}_{\lambda}^+(\mathcal{E}))$ carries a canonical normal coaction by $G$ induced by the natural one on $\mathcal{T}_{\lambda}^+(\mathcal{E})$. By \cite[Corollary 3.5]{Seh22} the canonical quotient $\mathcal{T}_{\lambda}(\mathcal{E}) \rightarrow C^*_r(([A \times_{\mathcal{E}} P]_g)_{g\in G})$ is completely isometric on $\mathcal{T}_{\lambda}^+(\mathcal{E})$, so that by the co-universal property of $C^*_\env(\mathcal{T}_{\lambda}^+(\mathcal{E}))$ we have a canonical quotient map $C^*_r(([A \times_{\mathcal{E}} P]_g)_{g\in G}) \rightarrow C^*_\env(\mathcal{T}_{\lambda}^+(\mathcal{E}))$. But now, since this quotient map is automatically equivariant with respect to the canonical \emph{normal} coactions on these C*-algebras, we get that this quotient map intertwines the natural \emph{faithful} conditional expectations on these two C*-algebras arising from the respective normal coactions. A standard conditional expectation argument together with \cite[Theorem 3.10]{Seh18} then shows that this canonical quotient map is injective. Corollary \ref{cor:main} then furnishes a complete proof of \cite[Theorem 5.1]{Seh22}.
\end{remark}

\section{C*-envelopes of operator algebras from categories}
In this section, we turn to left cancellative small categories as in \S~\ref{ss:prelimcats} and the C*-envelopes of their operator algebras.
Examples include left cancellative monoids, path categories of directed graphs, higher-rank graphs, and $P$-graphs.
In \S~\ref{ss:boundary-quotient}, we show that when $\fC$ is cancellative and the groupoid $I_l\ltimes \Omega$ is Hausdorff, the boundary quotient C*-algebra is a C*-cover for the operator algebra of $\fC$ (Theorem~\ref{thm:ciAr}),
and we provide sufficient conditions for when it is the C*-envelope.
This relies entirely on \'etale groupoid techniques.
In \S~\ref{ss:functors->coactions}, we show that a functor $\rho\colon \fC \to \fG$ from a left cancellative small category $\fC$ to a (discrete) groupoid $\fG$ induces a normal coaction on the operator algebra of $\fC$
and thus, by Theorem~\ref{thm:main}, a normal coaction on the C*-envelope.
In Theorem~\ref{thm:Pgraphs}, we apply this to compute the C*-envelope of the operator algebra of finitely aligned higher-rank graphs and $P$-graphs.
Finally, in \S~\ref{ss:groupoid-embeddable}, we consider the groupoid-embeddable case, where we identify the C*-envelope with the boundary quotient, thus answering Li's Question~\ref{ques:Li} for all such categories.

\subsection{The boundary quotient is a C*-cover}
\label{ss:boundary-quotient}
Throughout this subsection, we assume that $\fC$ is a \emph{cancellative} small category and that $I_l\ltimes\Omega$ is Hausdorff. 
Inside the groupoid C*-algebra $C^*_r(I_l\ltimes \Omega)$ there is a natural operator subalgebra 
\[
\A_r(\fC)\coloneq\overline{\alg}(\{1_{[c,\Omega(\mfd(c)\fC))]} : c\in \fC\})\subseteq C^*_r(I_l\ltimes \Omega).
\]
When $I_l\ltimes\Omega$ is Hausdorff, the *-isomorphism from \eqref{eqn:Jackmap} restricts to a completely isometric isomorphism between $\A_r(\fC)$ and the operator algebra $\A_\lambda(\fC)$ from Definition~\ref{def:A(fC)}.
We will show that the boundary quotient map $q_\partial\colon C_r^*(I_l\ltimes\Omega)\to C_r^*(I_l\ltimes\partial\Omega)$ from \eqref{eqn:bdryqmap} is completely isometric on $\A_r(\fC)$. Our strategy is inspired by the proof of \cite[Lemma~3.2]{DK20}. For each $\chi\in\partial\Omega$ and each subset $X\subseteq \fC$, put 
\[
[X,\chi]\coloneq\{[c,\chi] : c\in X, \chi(\mfd(c)\fC)=1\}\subseteq (I_l\ltimes\partial\Omega)_\chi.
\]
Since $\fC$ is cancellative, we have $[c,\chi]=[d,\chi]$ if and only if $c=d$. Indeed, the equality $[c,\chi]=[d,\chi]$ implies that there exists $Y\in\chi^{-1}(1)$ such that $cy =dy$ for $y\in Y$ and, by right cancellation, this implies that $c=d$. 

\begin{lemma}
For each $\chi\in\partial\Omega$, there is a *-homomorphism
    \[
        \vartheta_\chi\colon C_r^*(I_l\ltimes\Omega)\to \Bz(\ell^2([\fC,\chi]))
    \]
    determined by
    \begin{equation}
        \label{eqn:vartheta_chi}
    \vartheta_\chi(1_{[s,\Omega(\dom(s))]})e_{[d,\chi]}=\begin{cases}
        e_{[s(d),\chi]} & \text{ if } d\in\dom(s),\\
        0 & \text{ if } d\not\in\dom(s),
    \end{cases}
    \end{equation}   
    for all $s\in I_l^\times$ and $[d,\chi] \in [\fC,\chi]$.
\end{lemma}
\begin{proof}
For $\chi\in\partial\Omega$, let $\fC_\chi \coloneq \{c\in \fC : \chi(\mfd(c)\fC)=1\}$. It is easy to verify directly that $\fC_\chi$ is an $I_l$-invariant subset of $\fC$, so that $\ell^2(\fC_\chi)\subseteq \ell^2(\fC)$ is a $C_\lambda^*(\fC)$-reducing subspace. Denote by $\mu$ the associated representation of $C_\lambda^*(\fC)$ on $\ell^2(\fC_\chi)$, and let $U\colon \ell^2(\fC_\chi)\to \ell^2([\fC,\chi])$ be the unitary given by $Ue_c = e_{[c,\chi]}$ for all $c\in \fC_\chi$. 
Define $\vartheta_\chi\colon C_r^*(I_l\ltimes\Omega)\to \Bz(\ell^2([\fC,\chi]))$ by $\vartheta_\chi(a) \coloneq U\mu(\mfj(a)) U^*$, where $\mfj$ is the *-isomorphism from \eqref{eqn:Jackmap}. 
A straightforward computation shows that $\vartheta_\chi$ satisfies \eqref{eqn:vartheta_chi}.
\end{proof}

If $\fC$ is a cancellative \emph{monoid}, then $\vartheta_\chi$ is unitarily equivalent to the map $\mfj$ from \eqref{eqn:Jackmap}, and is therefore injective. In general, we only get injectivity by taking the sum of all such representations.

\begin{lemma}
\label{lem:vartheta}
The *-homomorphism
\[
\vartheta\coloneq\oplus_{\chi\in\partial\Omega}\vartheta_\chi\colon C_r^*(I_l\ltimes\Omega)\to \Bz\left(\textstyle{\bigoplus}_{\chi\in\partial\Omega}\ell^2([\fC,\chi])\right)
\]
is injective.
\end{lemma}
\begin{proof}
Consider the diagram 
  \begin{equation}
      \label{diag:expecting}
  \begin{tikzcd}
      C_r^*(I_l\ltimes\Omega)\arrow[d,"E_r"]\arrow[r,"\vartheta"] & \Bz\left(\bigoplus_{\chi\in\partial\Omega}\ell^2([\fC,\chi])\right) \arrow[d,"\prod_\chi\Phi_\chi"]\\
      C_0(\Omega) \arrow[r,"\vartheta\vert_{C_0(\Omega)}"]& \prod_{\chi\in\partial\Omega}\ell^\infty([\fC,\chi])\nospacepunct{,}
  \end{tikzcd}
  \end{equation}
  where $E_r\colon C_r^*(I_l\ltimes\Omega)\to C_0(\Omega)$ and $\Phi_\chi\colon \Bz(\ell^2([\fC,\chi]))\to\ell^\infty([\fC,\chi])$ for all $\chi \in \partial \Omega$ are the canonical faithful conditional expectations. We claim that diagram \eqref{diag:expecting} commutes. It is enough to check this on the spanning elements $1_{[s,\Omega(\dom(s))]}$ for all $s\in I_l^\times$. To show that, fix $\chi\in\partial\Omega$, and note that we have 
  \[
  \Phi_\chi(\vartheta_\chi(1_{[s,\Omega(\dom(s))]}))=1_{[\fix(s),\chi]}, 
  \]
  where $\fix(s)\coloneq\{x\in\dom(s) : s(x)=x\}$. On the other hand, we have that  $E_r(1_{[s,\Omega(\dom(s))]}))=1_{[s,\Omega(\dom(s))]\cap \Omega}$. By \cite[Proposition~3.14]{ExelPardo}, we get $[s,\Omega(\dom(s))]\cap \Omega=\cF_s$ where we define $\cF_s \coloneq \bigcup_{X\in \cJ,X\subseteq \fix(s)}\Omega(X)$. 
Since $I_l\ltimes\Omega$ is Hausdorff, \cite[Lemma 4.1(i)]{LiGarI} implies that there exists a finite (possibly empty) subset $F\subseteq\cJ^\times$ such that $\fix(s)=\bigcup_{X\in F}X$, so that $\cF_s=\bigcup_{X\in F}\Omega(X)$.
By the inclusion-exclusion principle, we have that
\[
1_{\cF_s}=1_{\bigcup_{X\in F}\Omega(X)}=\textstyle{\sum}_{\emptyset\neq A\subseteq F}(-1)^{|A|-1}\textstyle{\prod}_{X\in A}1_{\Omega(X)}.
\]
Since $\vartheta_\chi(1_{\Omega(X)})=1_{[X,\chi]}$, this gives us 
\[
\vartheta_\chi(1_{\cF_s})=
\textstyle{\sum}_{\emptyset\neq A\subseteq F}(-1)^{|A|-1}\textstyle{\prod}_{X\in A}\vartheta_\chi(1_{\Omega(X)})=1_{\bigcup_{X\in F}[X,\chi]}=1_{[\fix(s),\chi]},
\]
so that $\vartheta_\chi(E_r(1_{[s,\Omega(\dom(s))]})) =1_{[\fix(s),\chi]}$. Thus, diagram \eqref{diag:expecting} commutes.

Since both $E_r$ and $\prod_\chi\Phi_\chi$ are faithful, in order to prove $\vartheta$ is injective, it will suffice to prove that $\vartheta\vert_{C_0(\Omega)}$ is injective.
By \cite[Proposition 5.6.21]{CEL}, the kernel of $\vartheta\vert_{C_0(\Omega)}$ is generated by the projections $1_{\Omega(X)}-1_{\bigcup_{i=1}^n\Omega(X_i)}$, where $X,X_1,\ldots,X_n\in\cJ$ with $X_i\subseteq X$ are such that $[X,\chi]=\bigcup_{i=1}^n[X_i,\chi]$ for all $\chi\in\partial\Omega$. Thus, it will suffice to show that $\vartheta_{\chi}$ is nonzero on all such projections. Fix such a projection $1_{\Omega(X)}-1_{\bigcup_{i=1}^n\Omega(X_i)}$ and note that it is equal to zero if and only if $X = \bigcup_{i=1}^nX_i$. So we need only prove that $X = \bigcup_{i=1}^nX_i$ whenever we have $[X,\chi]=\bigcup_{i=1}^n[X_i,\chi]$ for every $\chi \in \partial \Omega$. 
Take $x\in X$, and choose $\chi\in\Omega_\max$ such that $\chi(\mfd(x)\fC)=1$ (such a $\chi$ exists by \cite[Lemma~2.21(ii)]{LiGarI}). Then, since $[x,\chi]\in [X,\chi]$ and we have the equality $[X,\chi]=\bigcup_{i=1}^n[X_i,\chi]$, we get that $[x,\chi]\in [X_i,\chi]$ for some $i$, so $x\in X_i$ by right cancellation. 
We conclude that $\vartheta\vert_{C_0(\Omega)}$ is injective.
\end{proof}

We are now ready to prove that  $C_r^*(I_l\ltimes\partial\Omega)$ is a C*-cover of $\A_r(\fC)$.

\begin{theorem}   
\label{thm:ciAr}
Let $\fC$ be a cancellative small category, and assume $I_l\ltimes\Omega$ is Hausdorff. Then, the quotient map $q_\partial\colon C_r^*(I_l\ltimes\Omega)\to C_r^*(I_l\ltimes\partial\Omega)$ from \eqref{eqn:bdryqmap} is completely isometric on $\A_r(\fC)$. 
\end{theorem}

\begin{proof}
Since $(I_l\ltimes\Omega)_\chi=(I_l\ltimes\partial\Omega)_\chi$ for all $\chi\in\partial\Omega$, we have $q_\partial = \oplus_{\chi\in\partial\Omega}\rho_\chi$.
For each $\chi\in\partial\Omega$, let $P_\chi\colon \ell^2((I_l\ltimes\Omega)_\chi))\to \ell^2([\fC,\chi])$ be the orthogonal projection associated with the inclusion $[\fC, \chi] \subseteq (I_l\ltimes\Omega)_\chi$. 
A straightforward calculation shows that 
  \begin{equation}
  \label{eqn:rhochi}
  \rho_\chi(1_{[c,\Omega(\mfd(c)\fC)]})e_{[d,\chi]}=\begin{cases}
      e_{[cd,\chi]} & \text{ if } \mfd(c)=\mft(d),\\
     0 & \text{ if }\mfd(c)\neq \mft(d),
  \end{cases}
  \end{equation}
  for all $c\in\fC$ and $[d,\chi]\in [\fC,\chi]$. Hence, the subspace $\ell^2([\fC,\chi])$ is invariant under $\rho_\chi(1_{[c,\Omega(\mfd(c)\fC)]})$, and $P_\chi\rho_\chi(1_{[c,\Omega(\mfd(c)\fC)]})P_\chi=\vartheta_\chi(1_{[c,\Omega(\mfd(c)\fC)]})$. It follows that $\bigoplus_{\chi\in\partial\Omega}\ell^2([\fC,\chi])$ is invariant under $\A_r(\fC)$, so we obtain an algebra homomorphism 
  $\psi\colon \A_r(\fC)\to \Bz\left(\textstyle{\bigoplus}_{\chi\in\partial\Omega}\ell^2([\fC,\chi])\right)$ 
  given by
  \[
  1_{[c,\Omega(\mfd(c)\fC))]}\mapsto \textstyle{\bigoplus}_{\chi\in\partial\Omega}P_\chi\rho_\chi(1_{[c,\Omega(\mfd(c)\fC)]})P_\chi,
  \]
  for all $c\in \fC$.
  The injective *-homomorphism $\vartheta$ from Lemma~\ref{lem:vartheta} satisfies $\vartheta(1_{[c,\Omega(\mfd(c)\fC)]})=\psi(1_{[c,\Omega(\mfd(c)\fC))]})$ for all $c\in\fC$, and this implies that $\psi$ is completely isometric. Since $\psi$ is the compression of $q_\partial$ by the projection $\oplus_{\chi\in\partial\Omega}P_\chi$, we see that $q_\partial$ is also completely isometric.
  \end{proof}

The general strategy for the proof of Theorem~\ref{thm:ciAr} is inspired by \cite[Lemma~3.2]{DK20}, but the proof differs significantly at the technical level and is in the language of groupoids rather than product systems. Additionally, we do not have an analogue of Fowler's Theorem (\cite[Theorem~7.2]{Fow02}), which played a crucial role in the proof of \cite[Lemma~3.2]{DK20}. This is why we needed to establish the faithfulness result in Lemma~\ref{lem:vartheta}.

For submonoids of groups, Theorem~\ref{thm:ciAr} provides a new approach to proving that the boundary quotient C*-algebra is a C*-cover. 
Moreover, Theorem~\ref{thm:ciAr} covers new examples of monoids, e.g., every cancellative monoid whose groupoid is Hausdorff. 
This class includes, e.g., all cancellative finitely aligned monoids. 
Even the class of singly aligned (right LCM) cancellative monoids contains interesting examples that are not group-embeddable (see Section \ref{s:right-LCM} below).

The co-universal property of the C*-envelope now provides us with the following corollary.

\begin{corollary}
\label{cor:lcscC*cover}
Let $\fC$ be a cancellative small category, and assume $I_l\ltimes\Omega$ is Hausdorff.
Then, there exists a surjective *-homomorphism 
    \begin{equation}
\label{eqn:pi}
    \pi_\env\colon C_r^*(I_l\ltimes\partial\Omega)\to C_\env^*(\A_r(\fC))
    \end{equation}
    such that $\pi_\env(1_{[c,\partial\Omega(\mfd(c)\fC))]})=1_{[c,\Omega(\mfd(c)\fC))]}$ for all $c\in \fC$.
\end{corollary}

We observe that $\pi_\env$ from \eqref{eqn:pi} is always injective on the canonical diagonal subalgebra.

\begin{lemma}
\label{lem:piinjondiagonal}
Assume $\fC$ is a cancellative small category and that $I_l\ltimes\Omega$ is Hausdorff. Then, the map $\pi_\env$ from \eqref{eqn:pi} is injective on $C_0(\partial\Omega)$.
\end{lemma}
\begin{proof}
Let $K\subseteq\partial\Omega$ be a closed $I_l$-invariant subset with $\ker(\pi_\env\vert_{C_0(\partial\Omega)})=C_0(\partial\Omega\setminus K)$. We need to show that $K=\partial\Omega$. For each $c\in\fC$, we have 
\[
\pi_\env(1_{\partial\Omega(c\fC)})=\pi_\env(1_{[c,\partial\Omega(\mfd(c)\fC)]}1_{[c,\partial\Omega(\mfd(c)\fC))]}^*)=1_{[c,\Omega(\mfd(c)\fC))]}1_{[c,\Omega(\mfd(c)\fC))]}^*\neq 0,
\]
so that $K\cap \partial\Omega(c\fC)\neq\emptyset$, and we may choose some $\gamma_c\in K\cap \partial\Omega(c\fC)$ for every $c\in \fC$.

If $0\not\in I_l$, then $\partial\Omega=\{\chi_\infty\}$ by Remark~\ref{rmk:bdry}, and we are done. So assume $0\in I_l$. 
Since $K$ is closed, it suffices to show that $\{\gamma_c : c\in\fC\}$ is dense in $\partial \Omega$. 
Let $\partial\Omega(X;\mff)$ be a nonempty basic open subset of $\partial \Omega$ as in \S~\ref{ss:prelimcats}. 
Since $0\in I_l$, there exists a nonempty constructible right ideal $Z$ of $X$ with $Z\cap \bigcup_{Y\in\mff}Y=\emptyset$ (see the discussion after Remark~\ref{rmk:bdry}). 
Take $c\in Z$, and note that $c\fC\subseteq Z \subseteq X$ and $c\fC\cap Y=\emptyset$ for all $Y\in\mff$. Thus, $\gamma_c(X)=1$ and $\gamma_c(Y)=0$ for all $Y\in\mff$, so that $\gamma_c\in \partial\Omega(X;\mff)$. Thus, we have shown that $\{\gamma_c : c\in\fC\}$ is dense in $\partial \Omega$, which implies that $K = \partial \Omega$.
\end{proof}

Recall that an inclusion of C*-algebras $\cD \subseteq \cB$ \emph{detects ideals} if every nontrivial ideal $\cI \subseteq \cB$ intersects $\cD$ nontrivially. 
See \cite[Theorem~7.2]{KKLRU} for a characterization of when the diagonal detects ideals in a groupoid C*-algebra.
In the setting of Corollary~\ref{cor:lcscC*cover}, if $C_0(\partial\Omega)$ detects ideals in $C_r^*(I_l\ltimes\partial\Omega)$ (e.g., if the groupoid $I_l\ltimes\partial\Omega$ is effective, see \cite[Corollary~5.13]{LiGarI} for a characterization), then $\pi_\env$ is injective by Lemma~\ref{lem:piinjondiagonal}.
Moreover, if $I_l\ltimes\Omega$ is second countable, Hausdorff, and $0\in I_l$, then  $\pi_\env$ is injective if and only if its restriction to the C*-subalgebra of the subgroupoid of the interior of the isotropy is injective, see \cite[Theorem~3.1(a)]{BNRSW16}.  
An explicit description of this subgroupoid can be derived from \cite[\S~5]{LiGarI}. 

\subsection{From functors to coactions}
\label{ss:functors->coactions}
In this subsection, we use coaction techniques and Theorem~\ref{thm:main} to find significantly broader sufficient conditions for the map $\pi_\env$ from \eqref{eqn:pi} to be injective.

A (discrete) groupoid $\fG$ admits a universal group $\cU(\fG)$ which comes with a functor (i.e., a groupoid homomorphism) $j_\fG\colon \fG \to \cU(\fG)$ that is characterized by the following universal property: whenever $\mu\colon \fG \to H$ is a functor to a group, there exists a group homomorphism $\mu'\colon \cU(\fG) \to H$ such that $\mu'\circ j_\fG = \mu$. Our first task is to construct normal coactions of $\cU(\fG)$ on the C*-algebras attached to $\fC$. Let $I_l^\times \coloneq I_l\setminus \{0\}$ and $\cJ^\times\coloneq \cJ\setminus\{\emptyset\}$.

There always exists a functor $\rho\colon \fC \to \fG$ to a groupoid $\fG$. For instance, one may take $\fG$ to be the enveloping groupoid associated to $\fC$ (see \cite[\S~II.3.1]{Doh15}).
Note that $\rho$ need not be injective, and that in some cases $\fG$ may be a group.
\begin{lemma}
\label{lem:partialhom}
Suppose $\fC$ is a left cancellative small category and that $\rho\colon \fC \to \fG$ is a functor to a groupoid $\fG$.
Then, there exists a map  $\tilde{\rho}\colon I_l^\times\to \fG$ such that $\tilde{\rho}(st) = \tilde{\rho}(s) \tilde{\rho}(t)$ for all $s,t\in I_l$ with $st \neq 0$ and $\rho(s(x))=\tilde{\rho}(s)\rho(x)$ for all $s\in I_l^\times$ and $x\in\dom(s)$.
In particular, we have $\tilde{\rho}(c)=\rho(c)$ for all $c\in\fC$.
\end{lemma}
\begin{proof}
Let $d\in\fC$. For every $x\in d\fC$, we can write $x=dd'$ for some $d'\in \mfd(d)\fC$, and then $d^{-1}(x)=d'$. Thus,
\[
\rho(d^{-1}(x))=\rho(d')=\rho(d)^{-1}\rho(d)\rho(d')=\rho(d)^{-1}\rho(dd')=\rho(d)^{-1}\rho(x).
\]
We also have $\rho(cx)=\rho(c)\rho(x)$ for all $c\in\fC$ and $x\in\mfd(c)\fC$. Thus, given $s=d_1^{-1}c_1\cdots d_n^{-1}c_n\in I_l^\times$, where $c_i,d_i\in\fC$, we have by induction
\[
\rho(s(x))=\rho(d_1)^{-1}\rho(c_1)\cdots\rho(d_n)^{-1}\rho(c_n)\rho(x)
\]
for every $x\in\dom(s)$. 
By right cancellation in the groupoid $\fG$, this gives us a well-defined map $\tilde{\rho}\colon I_l^\times\to \fG$ such that  
\[
\tilde{\rho}(d_1^{-1}c_1\cdots d_n^{-1}c_n)=\rho(d_1)^{-1}\rho(c_1)\cdots \rho(d_n)^{-1}\rho(c_n)
\]
for all $c_i,d_i\in\fC$ with $d_1^{-1}c_1\cdots d_n^{-1}c_n\in I_l^\times$. It is clear that $\tilde{\rho}$ satisfies the stated conditions.
\end{proof}

Given $\tilde{\rho}$ as in Lemma~\ref{lem:partialhom}, we put $\bar{\rho}\coloneq j_\fG\circ\tilde{\rho}\colon I_l^\times\to \cU(\fG)$. Then, $\bar{\rho}$ is a partial homomorphism in the sense that $\bar{\rho}(st)=\bar{\rho}(s)\bar{\rho}(t)$ for all $s,t\in I_l^\times$ with $st\neq 0$.

\begin{lemma}
\label{lem:kappa_rho}
Suppose $\fC$ is a left cancellative small category and that $\rho\colon \fC \to \fG$ is a functor to a groupoid $\fG$. 
Then, there is a continuous groupoid homomorphism
    \begin{equation}
    \label{eq:kappa_rho}
    \kappa_\rho\colon I_l\ltimes\Omega\to \cU(\fG)
    \end{equation}
    such that $\kappa_\rho([s,\chi])=\bar{\rho}(s)$ for all $[s,\chi]\in I_l\ltimes\Omega$.
\end{lemma}

\begin{proof}
If $[s,\chi]=[t,\chi]$ in $I_l\ltimes\Omega$, then there exists $X\in\chi^{-1}(1)$ with $s\circ \id_X=t \circ \id_X\neq 0$, so that $\bar{\rho}(s) = \bar{\rho}(s)\bar{\rho}(\id_X)= \bar{\rho}(s \circ \id_X) = \bar{\rho}(t \circ \id_X) = \bar{\rho}(t)$. 
Thus, there is a well-defined map $\kappa_\rho\colon I_l\ltimes\Omega\to \cU(\fG)$ such that $\kappa_\rho([s,\chi])=\bar{\rho}(s)$ for all $[s,\chi]\in I_l\ltimes\Omega$. 
If $\chi \in \Omega$ and $t,s\in I_l$ with $[t,s.\chi] [s,\chi] = [ts, \chi]$, then $ts\neq 0$ and 
\[
\kappa_\rho([t,s.\chi])\kappa_\rho([s,\chi]) = \bar{\rho}(t)\bar{\rho}(s) = \bar{\rho}(ts) = \kappa_\rho([ts, \chi]).
\]
This shows that $\kappa_\rho$ is a groupoid homomorphism. Given $g\in \cU(\fG)$, we have
\[
\kappa_\rho^{-1}(g)=\bigcup\{[s,U] : s\in \bar{\rho}^{-1}(g), U\subseteq\Omega(\dom(s))\text{ open}\},
\]
which is open, so $\kappa_\rho$ is continuous.
\end{proof}

Denote by $\partial\kappa_\rho$ the restriction of $\kappa_\rho$ to $I_l\ltimes\partial\Omega$. Next, we use an observation from \cite[Lemma~6.1]{CRST21}.

\begin{proposition}
\label{prop:coactions}
Suppose $\fC$ is a left cancellative small category and that $\rho\colon \fC \to \fG$ is a functor to a groupoid $\fG$.
Then, there are reduced coactions $\delta_{\rho,\lambda}\colon \cU(\fG)\coacts C_r^*(I_l\ltimes\Omega)$ and $\partial\delta_{\rho,\lambda} \colon \cU(\fG)\coacts C_r^*(I_l\ltimes\partial\Omega)$ determined on generators by $1_{[s,\Omega(\dom(s))]}\mapsto 1_{[s,\Omega(\dom(s))]}\otimes \lambda_{\bar{\rho}(s)}$ and $1_{[s,\partial\Omega(\dom(s))]}\mapsto 1_{[s,\partial\Omega(\dom(s))]}\otimes \lambda_{\bar{\rho}(s)}$ for all $s\in I_l^\times$.
\end{proposition}
\begin{proof}
 By \cite[Lemma~6.1]{CRST21}, there is a *-homomorphism $\delta_{\rho,\lambda}\colon C_r^*(I_l\ltimes\Omega) \rightarrow C_r^*(I_l\ltimes\Omega)\otimes \cU(\fG)$, which is nondegenerate, satisfies the coaction identity, and such that $\delta_{\rho,\lambda}(f)=f\otimes \lambda_g$ for $f\in C_c(I_l\ltimes\Omega)$ with $\supp(f)\subseteq \kappa_\rho^{-1}(g)$ and $g\in \fG$. Moreover, from the proof of \cite[Lemma~6.1]{CRST21}, we see that $\delta_{\rho,\lambda}$ is injective, so that $\delta_{\rho,\lambda}$ is a reduced coaction. Existence of $\partial\delta_{\rho,\lambda}$ is proven the same way by using $\partial\kappa_\rho$ on $I_l\ltimes\partial\Omega$ instead of $\kappa_{\rho}$.
\end{proof}

\begin{corollary}
\label{cor:coactionEnvA(C)}
Suppose $\fC$ is a left cancellative small category and that $\rho\colon \fC \to \fG$ is a functor to a groupoid $\fG$.
Then, there is a normal coaction $\delta_\rho\colon\cU(\fG)\coacts \A_r(\fC)$ which extends to a normal coaction $\delta_{\rho,\env}$ on $ C_\env^*(\A_r(\fC))$
such that $\delta_{\rho,\env}(1_{[c,\Omega(\mfd(c)\fC)]})=1_{[c,\Omega(\mfd(c)\fC)]}\otimes u_{\bar{\rho}(c)}$ for all $c\in\fC$.
\end{corollary}

\begin{proof}
From \cite[Proposition 3.4]{DKKLL}, the reduced coaction $\delta_{\rho,\lambda}$ of Proposition~\ref{prop:coactions} induces a normal coaction $\delta_{\rho} \colon \cU(\fG) \coacts C_r^*(I_l\ltimes\Omega)$ such that $\delta_{\rho} = (\id \otimes \lambda) \circ\delta_{\rho,\lambda}$.
Since $\delta_\rho$ is normal, the restriction of $\delta_\rho$ to $\A_r(\fC)$ is also a normal coaction. 
The existence of $\delta_{\rho,\env}$ now follows from Theorem~\ref{thm:main}.
\end{proof}

Suppose $\rho\colon \fC \to \fG$ is a functor to a groupoid $\fG$, and recall the terminology from \S~\ref{ss:functors->coactions}.
By \cite[Lemma~6.3]{CRST21}, there is a canonical isomorphism $C_r^*(I_l\ltimes\partial\Omega)^{\partial \delta_\rho}\cong C_r^*(\ker(\partial\kappa_\rho))$,
and we view this as a C*-subalgebra of $C_r^*(I_l\ltimes\partial\Omega)$. Theorem~\ref{thm:main} now gives us a characterization of injectivity of $\pi_\env$ in terms of $C_r^*(\ker(\partial\kappa_\rho))$.

\begin{proposition}
Assume $\fC$ is a cancellative small category and that $I_l\ltimes\Omega$ is Hausdorff. Let $\rho\colon \fC\to\fG$ be a functor to a groupoid $\fG$. 
Then, the map $\pi_\env$ from \eqref{eqn:pi} is injective if and only if its restriction to $C_r^*(\ker(\partial\kappa_\rho))$ is injective.
\end{proposition}
\begin{proof}
Consider the diagram
\[
\begin{tikzcd}
\label{eqn:denv}
 C_r^*(I_l\ltimes\partial\Omega)\arrow[d,"\partial E"]\arrow[rr,"\pi_\env"]    && C_\env^*(\A_r(\fC)) \arrow[d,"\Psi"]\\
C_r^*(\ker(\partial\kappa_\rho)) \arrow[rr,"\pi_\env\vert_{C_r^*(\ker(\partial\kappa_\rho))}"] && C_\env^*(\A_r(\fC))_e\nospacepunct{,}
\end{tikzcd}
\]
where $\partial E$ is the faithful conditional expectation associated with the coaction $\partial\delta_\rho$, and $\Psi$ is the faithful conditional expectation onto the unit fibre associated with the normal coaction $\delta_{\rho,\env}$ from Corollary~\ref{cor:coactionEnvA(C)}. This diagram commutes, so $\pi_{\env}$ is injective if and only if $\pi_\env\vert_{C_r^*(\ker(\partial\kappa_\rho))}$ is injective. 
\end{proof}

If $\ker(\partial\kappa_\rho)$ is second countable, then \cite[Theorem~3.1(a)]{BNRSW16} implies that $\pi_\env$ is injective on $C_r^*(\ker(\partial\kappa_\rho))$ if and only if it is injective on the C*-subalgebra of the interior of the isotropy subgroupoid of $\ker(\partial\kappa_\rho)$.
Since $\pi_\env$ is always injective on $C_0(\partial\Omega)$ by Lemma~\ref{lem:piinjondiagonal}, we obtain the following sufficient condition for $\pi_{\env}$ to be injective.

\begin{corollary}
\label{cor:intersectionprop}
Assume $\fC$ is a cancellative small category and that $I_l\ltimes\Omega$ is Hausdorff.
 If $C_0(\partial\Omega)$ detects ideals in $C_r^*(\ker(\partial\kappa_\rho))$ (e.g. if $\ker(\partial\kappa_\rho)$ is effective), then the map $\pi_\env$ from \eqref{eqn:pi} is injective.
 \end{corollary}

Let us show how Corollary~\ref{cor:intersectionprop} can be used to compute C*-envelopes for operator algebras arising from finitely aligned higher-rank graphs, and even $P$-graphs.
We include this to show the breadth of our results, though we expect Theorem~\ref{thm:Pgraphs} to be covered (using different methods) from Sehnem's work \cite[Theorem~5.1]{Seh22}.

Let $P$ be a submonoid of a group $G$ such that $P\cap P^{-1}=\{e\}$.
A \emph{$P$-graph} in the sense of \cite[Definition~8.1]{OP} (cf. \cite[Definition~6.1]{RW17} and \cite[Definition~2.1]{Brownlowe-Sims-Vittadello}) is a pair $(\fC,\deg)$, where $\fC$ is a finitely aligned countable category, and $\deg\colon\fC\to P$ is a functor satisfying the unique factorization property:
for every $c\in\fC$ and $p_1,p_2\in P$ with $\deg(c)=p_1p_1$, there exist unique $c_1,c_2\in\fC$ with $\mfd(c_1)=\mft(c_2)$ such that $\deg(c_1)=p_1$, $\deg(c_2)=p_2$, and $c=c_1c_2$.
If $(\fC,\deg)$ is a $P$-graph, then $\fC$ is cancellative and the units $\fC^0$ are the only invertible elements in $\fC$.
Thus, if $(\fC,\deg)$ is a finitely aligned $P$-graph, then $I_l\ltimes \Omega$ is Hausdorff by \cite[Corollary~4.2]{LiGarI}.

Given a category $\fC$, we let $\Env(\fC)$ be the enveloping groupoid of $\fC$ in the sense of \cite[Definition~II.3.3]{Doh15}, and let $\rho_u\colon\fC\to \Env(\fC)$ be the canonical functor (note that $\Env(\fC)$ is called the fundamental groupoid in \cite{BKQ}). The pair $(\Env(\fC),\rho_u)$ is uniquely determined up to isomorphism by the following universal property: for every functor $F\colon\fC\to \cH$, where $\cH$ is a groupoid, there exists a unique functor $\widetilde{F}\colon \Env(\fC)\to \cH$ such that $F=\rho_u\circ\widetilde{F}$.

\begin{theorem}
\label{thm:Pgraphs}
    Let $(\fC,\deg)$ be a finitely aligned $P$-graph, where $P$ is a group-embeddable monoid. Then, the map $\pi_\env\colon C_r^*(I_l\ltimes\partial\Omega)\to C^*_\env(\A_r(\fC))$ from \eqref{eqn:pi} is a *-isomorphism.
\end{theorem}
\begin{proof}
We show that $\ker(\partial\kappa_{\rho_u})$ is a principal groupoid (i.e., every element whose range and source coincide is a unit).
It follows that $C_0(\partial\Omega)$ detects ideals in $\ker(\partial\kappa_{\rho_u})$, so the conclusion would then follow from Corollary~\ref{cor:intersectionprop}.

Put $\partial\kappa\coloneq \partial\kappa_{\rho_u}$ and suppose $P$ embeds into a group $G$ with identity element $e$. 
Viewing $\deg$ as a functor from $\fC$ to $G$, the universal property of $\Env(\fC)$ implies the existence of a functor $\rho_u'\colon \Env(\fC)\to G$ such that the following diagram commutes:
\begin{equation}
 \begin{tikzcd}
 \label{diag:deg}
        \fC\arrow[rd,"\deg"'] \arrow[r,"\rho_u"]& \Env(\fC)\arrow[d,"\rho_u'"] \\
         & G\nospacepunct{.}
\end{tikzcd}
\end{equation}
Every element of $I_l\ltimes\partial\Omega$ can be written as $[cd^{-1},\chi]$ for some $c,d\in\fC$ with $\mfd(c)=\mfd(d)$ and $\chi\in\partial\Omega$ with $\chi(d\fC)=1$ (see, e.g., \cite[Lemma~4.8]{OP}). Take such an element $[cd^{-1},\chi]$ and suppose $\partial\kappa([cd^{-1},\chi])=e$. 
This means that $\rho_u(c)=\rho_u(d)$, so by commutativity of \eqref{diag:deg}, we have $\deg(c)=\deg(d)$. 
Assuming that $[cd^{-1},\chi]$ is isotropy, i.e., that $cd^{-1}.\chi=\chi$, we have $1=\chi(d\fC)=cd^{-1}.\chi(d\fC)=\chi(dc^{-1}(c\fC\cap d\fC))$. In particular, $c\fC\cap d\fC\neq\emptyset$, and \cite[Lemma~8.2]{OP} implies that  $c=d$.
Therefore, $[cd^{-1},\chi]$ is a unit, and this shows that $\ker(\partial\kappa)$ is principal.
\end{proof}

This result covers all finitely aligned higher-rank graphs.
It is interesting to note that even among 2-graphs with a single vertex there are examples of cancellative monoids that are not group-embeddable.
See, e.g., \cite[Example~7.1]{PQR04} and \cite[Example~11.13]{LV20}.

In the next subsection, we showcase another natural class of examples where $\pi_\env$ is injective, which are not covered by the class of product systems over group-embeddable monoids as in \cite{Seh22}.

\subsection{Groupoid-embeddable categories}
\label{ss:groupoid-embeddable}
Let $\fG$ be a discrete groupoid with range and source maps $\mfr$ and $\mfs$, respectively. 
By \cite[Theorem~6.10]{BKQ}, the universal group $\cU(\fG)$ of $\fG$ can be described as follows. 
Let $[u]\coloneq \mfr(\mfs^{-1}(u))$ denote the orbit of a unit $u\in \fG^0$ and let $\cR$ be a complete set of representatives for the set of orbits $\fG^{0}/\fG\coloneq\{[u]: u\in\fG^0\}$, so that $\fG=\bigsqcup_{u\in\cR}\fG_{[u]}$, where $\fG_{[u]}\coloneq\{g\in\fG : \mfs(g)\in [u]\}$. For $u,v\in\fG^0$, we put $\fG_u^v\coloneq \{g\in\fG : \mfs(g)=u,\mfr(g)=v\}$, and for each $u\in\cR$, we let $\cX_u\coloneq [u]\setminus\{u\}$. Then, \cite[Theorem~6.10]{BKQ} provides us with a group isomorphism
\[
\cU(\fG)\cong\Asterisk_{u\in\cR}\cU(\fG_{[u]}) \cong\Asterisk_{u\in\cR}(\Fz(\cX_u)\Asterisk \fG_u^u),
\] 
where $\Fz(\cX_u)$ is the free group on $\cX_u$. Moreover, the homomorphism $j_\fG\colon \fG\to \Asterisk_{u\in\cR}(\Fz(\cX_u)\Asterisk \fG_u^u)$ can be described as follows (see \cite[Proposition~6.8]{BKQ}).
For each $u\in\cR$ and $v\in [u]$, choose $\gamma_v\in \fG_u^v$. Then, 
\begin{equation}
\label{eqn:jmapBKQ}    
j_\fG(g)= \bar{z}(\gamma_z^{-1}g\gamma_y)\bar{y}^{-1},
\end{equation}
for all $g\in \fG_y^z$,
where $\bar{z},\bar{y}$ are the images of $z,y$ in $\Fz(\cX_u)$ for the unique element $u$ in $\cR$ such that $z,y\in [u]$. 
Note that $j_\fG$ generally depends on the various choices made above.
We shall identify $\cU(\fG)$ with $\Asterisk_{u\in\cR}(\Fz(\cX_u)\Asterisk \fG_u^u)$ via the isomorphism in equation \eqref{eqn:jmapBKQ}. 
Let $e$ be the identity element in $\cU(\fG)$.

\begin{lemma}
\label{lem:maptounivgroup}
Let $\fG$ be a discrete groupoid. Then, $j_\fG^{-1}(e)=\fG^0$, and $j_\fG$ is injective on $\fG\setminus\fG^0$.
\end{lemma}
\begin{proof}
Let $y,z\in\fG^0$ and suppose $j_\fG(g)=e$ for some $g\in \fG_y^z$. 
Then, $\gamma_z^{-1}g\gamma_y=\bar{z}^{-1}\bar{y}$ in the group $\Fz(\cX_u)\Asterisk\cU(\fG_u^u)$, where $u$ is the unique element in $\cR$ such that $z,y\in [u]$. Since  $\gamma_z^{-1}g\gamma_y\in\fG_u^u$ and $\bar{z}^{-1}\bar{y}\in\Fz(\cX_u)$, and the subgroups $\fG_u^u$ and $\Fz(\cX_u)$ intersect trivially in $\Fz(\cX_u)\Asterisk\fG_u^u$, we deduce that $\bar{z}^{-1}\bar{y} = e = \gamma_z^{-1}g\gamma_y$. 
In particular, 
$\bar{y}=\bar{z}$, so that $y=z$ and $g=\gamma_y\gamma_y^{-1}=y$ is a unit. This shows that $j_{\fG}^{-1}(e) \subseteq \fC^0$. 
The reverse containment follows from the definition of $j_{\fG}$.

Next, we show that $j_{\fG}$ is injective on $\fG \setminus \fG^0$. Suppose $g,h\in \fG\setminus\fG^0$ are such that $j_\fG(g)=j_\fG(h)$. Then, $g\in \fG_y^z$ and $h\in\fG_x^w$ for some $x,y,w,z\in\fG^0$, and we have $j_\fG(g)= \bar{z}(\gamma_z^{-1}g\gamma_y)\bar{y}^{-1}$ and $j_\fG(h)= \bar{w}(\gamma_w^{-1}h\gamma_x)\bar{x}^{-1}$.

If $\gamma_z^{-1}g\gamma_y$ is a unit, so that $g= \gamma_z \gamma_y^{-1}$, then $j_\fG(g) = \bar{z}\bar{y}^{-1}$ is in $\Fz(\cX_u)$. By uniqueness of reduced words in free products, $\bar{w}(\gamma_w^{-1}h\gamma_x)\bar{x}^{-1}=j_\fG(h)=j_\fG(g)=\bar{z}\bar{y}^{-1}$ implies that $\gamma_w^{-1}h\gamma_x$ is a unit, so that $h=\gamma_w\gamma_x^{-1}$. Now $j_\fG(g) = j_\fG(h)$ simplifies to $\bar{z}\bar{y}^{-1}=\bar{w}\bar{x}^{-1}$, which implies that $z=w$ and $y=x$. Hence, $g=\gamma_y \gamma_y^{-1}$ and $h=\gamma_x^{-1}\gamma_x$ are units, which contradicts our assumptions on $g,h$.

As both $\gamma_z^{-1}g\gamma_y$ and $\gamma_z^{-1}g\gamma_z$ are not units, the uniqueness of reduced words in free products implies that $\bar{z}(\gamma_z^{-1}g\gamma_y)\bar{y}^{-1}=\bar{w}(\gamma_w^{-1}h\gamma_x)\bar{x}^{-1}$, so $\bar{z}=\bar{w}$ and $\overline{y} = \overline{x}$.
Now we have $\gamma_z^{-1}g\gamma_y=\gamma_w^{-1}h\gamma_x$.
This forces $z=w$ and $y=x$, so we actually get that $\gamma_z^{-1}g\gamma_y=\gamma_z^{-1}h\gamma_y$. Multiplying this last equation on the left by $\gamma_z$ and on the right by $\gamma_y^{-1}$ yields $g=h$.
\end{proof}

Recall the definition of $\bar{\rho}\colon I_l^\times\to \cU(\fG)$ from the discussion after the proof of Lemma~\ref{lem:partialhom}, as well as our identification of $\cJ^\times$ with the nonzero idempotents in $I^{\times}_l$.

Note that a groupoid-embeddable category is automatically small and cancellative.

\begin{lemma}
    \label{lem:idpure}
Let $\fC$ be a small category and suppose $\rho\colon \fC\to\fG$ is an embedding into a discrete groupoid $\fG$. 
The partial homomorphism $\bar{\rho}$ is idempotent pure in the sense that $\bar{\rho}^{-1}(e)=\cJ^\times$.
\end{lemma}
\begin{proof}
Let $s\in I_l^\times$ with $\bar{\rho}(s)=e$. Then, $j_\fG\circ\tilde{\rho}(s)=e$, and by Lemma~\ref{lem:maptounivgroup} we get that $\tilde{\rho}(s)\in\fG^0$. 
Using Lemma~\ref{lem:partialhom}, we have that $\rho(s(x))=\tilde{\rho}(s)\rho(x)=\rho(x)$ for all $x\in\dom(s)$.
Since $\rho$ is injective, it follows that $s(x)=x$ for all $x\in\dom(s)$, so that $s\in\cJ^\times$.
\end{proof}

Applying \cite[Lemma~5.5.22]{CELY} (cf. also \cite{MS14}) yields the following consequence. 

\begin{proposition}
    \label{prop:idpure}
    Let $\fC$ be a small category and suppose $\rho\colon \fC\to\fG$ is an embedding into a discrete groupoid $\fG$. 
    Then, there is a canonical partial action of $\cU(\fG)$ on $\Omega$ and an isomorphism of topological groupoids 
    \begin{equation}
        \label{eqn:gpoidisom}
    I_l\ltimes\Omega\cong \cU(\fG)\ltimes\Omega,\quad [s,\chi]\mapsto (\bar{\rho}(s),\chi).
     \end{equation}
    In particular, $I_l\ltimes\Omega$ is Hausdorff.
\end{proposition}

Note that the isomorphism $I_l\ltimes\Omega\cong \cU(\fG)\ltimes\Omega$ from \eqref{eqn:gpoidisom} restricts to an isomorphism $I_l\ltimes\partial\Omega\cong \cU(\fG)\ltimes\partial\Omega$, so we obtain the following corollary, which extends \cite[Proposition 3.10]{Li:IMRN} from submonoids of groups to subcategories of groupoids.

\begin{corollary}
\label{cor:idpure}
Let $\fC$ be a small category that embeds in a groupoid $\fG$. Then, there is a canonical partial action of $\cU(\fG)$ on $\Omega$ and canonical *-isomorphisms 
\[
C_r^*(I_l\ltimes\Omega)\cong C_0(\Omega)\rtimes^r\cU(\fG) \quad\text{ and }\quad C_r^*(I_l\ltimes\partial\Omega)\cong C_0(\partial\Omega)\rtimes^r\cU(\fG).
\]
\end{corollary}

The following is the main result of this subsection, and answers Li's Question~\ref{ques:Li} in the affirmative for the class of groupoid-embeddable categories.

\begin{theorem}
\label{thm:gpoid-embeddable}
Let $\fC$ be a small category that embeds in a groupoid $\fG$. 
Then, the associated groupoid $I_l\ltimes\Omega$ is Hausdorff, and the map $\pi_\env\colon C_r^*(I_l\ltimes\partial\Omega)\to C^*_\env(\A_r(\fC))$ from \eqref{eqn:pi} is a *-isomorphism. That is, the C*-envelope $C^*_\env(\A_r(\fC))$ coincides with the boundary quotient C*-algebra $C_r^*(I_l\ltimes\partial\Omega)$.
\end{theorem}

\begin{proof}
Let $\rho\colon \fC \to \fG$ be an embedding into a groupoid $\fG$ so that $\fC$ is cancellative.
By Corollary~\ref{cor:idpure}, $I_l\ltimes\Omega$ is Hausdorff, and we have an isomorphism of topological groupoids $I_l\ltimes\partial\Omega\cong \cU(\fG)\ltimes\partial\Omega$ that carries $\partial\kappa_\rho$ to the homomorphism on $\cU(\fG)\ltimes\partial\Omega$ given by projecting onto $\cU(\fG)$.
Under this isomorphism, the coaction of $\cU(\fG)$ on $C_r^*(\cU(\fG)\ltimes\partial\Omega)$ from Proposition~\ref{prop:coactions} is carried to the canonical coaction of $\cU(\fG)$ on $C_r^*(\cU(\fG)\ltimes\partial\Omega)\cong C_0(\partial\Omega)\rtimes^r \cU(\fG)$, and the fixed point algebra of this latter coaction is $C_0(\partial\Omega)$ (see, e.g., \cite[Lemma~6.3]{CRST21}). 
By Lemma~\ref{lem:piinjondiagonal}, $\pi_\env$ is injective on $C_0(\partial\Omega)$, so by Corollary~\ref{cor:intersectionprop} we conclude that $\pi_\env$ is injective.
\end{proof}

\begin{examples}
By Ore's theorem, every left Ore category can be emdedded into a groupoid, see \cite[Proposition~II.3.11]{Doh15}. Special classes of Ore categories have been studied recently in the setting of Thompson's groups, see \cite{Wit19, SWZ19}, and  non-Ore groupoid-embeddable categories associated with graphs of groups have been studied in \cite{Hig} and \cite{LW17}. See, e.g., \cite{Jo08} for more on groupoid-embeddability of small categories.
\end{examples}

\section{Right LCM monoids} \label{s:right-LCM}

Recall that a monoid $P$ is right LCM if it is left cancellative and for all $c,d\in P$, we have that $cP\cap dP$ is either empty or of the form $rP$ for some $r\in P$. Right LCM monoids have been studied intensely by many operator algebraists, see, e.g., \cite{BRRW14,St15, ABLS19, BLRS20, LOS21, LS22, NS22, St22, KKLL22b}. 
In this section (Theorem~\ref{thm:lcm-monoids}), we identify the C*-envelope of all right LCM monoids. This includes examples which are not group embeddable, including the class of Malcev type monoids from \cite{EH+}, as well as a class of gcd-monoids introduced by Dehornoy in \cite{Doh17I, Doh17I} which contains examples which are not group-embeddable (see \cite[Proposition~4.3]{DW17}). Hence, our result applies to a large class of monoids that need not be group-embeddable, and are therefore not covered by previous works (for instance \cite{Seh22}).

Let $P$ be a cancellative right LCM monoid. Then, $I_l\ltimes\Omega$ is Hausdorff by \cite[Corollary~4.2]{LiGarI}. Following \cite[\S~4]{St15} (which generalizes \cite[Definition 5.4]{CL07}), we let
\[
P_0\coloneq\{c\in P : cP\cap dP\neq\emptyset\text{ for all }d\in P\}
\] 
be the \emph{core} submonoid of $P$. We let 
\[
I_l^0\coloneq\{cd^{-1}\in I_l : c,d\in P_0\},
\]
so that $I_l^0$ is a sub inverse semigroup of $I_l$, and $0\notin I_l^0$. Consider the groupoid 
\[
I_l^0\ltimes\partial\Omega\coloneq \{[cd^{-1},\chi]\in I_l\ltimes\partial\Omega : c,d\in P_0\}.
\]
Then $I_l^0\ltimes\partial\Omega$ is an open subgroupoid of $I_l\ltimes\partial\Omega$, so the canonical inclusion map $C_c(I_l^0\ltimes\partial\Omega)\to C_c(I_l\ltimes\partial\Omega)$ extends to an injective *-homomorphism $C_r^*(I_l^0\ltimes\partial\Omega)\to C_r^*(I_l\ltimes\partial\Omega)$. We henceforth identify $C_r^*(I_l^0\ltimes\partial\Omega)$ with its image in $C_r^*(I_l\ltimes\partial\Omega)$. Since $P_0$ is right Ore (left reversible), the canonical homomorphism $\rho_0\colon P_0\to \cU(P_0)$ is injective, and $\cU(P_0)=\{\rho_0(c)\rho_0(d)^{-1} :c,d\in P\}$.

\begin{lemma}
\label{lem:kappa_0}
There is a continuous groupoid homomorphism
\begin{equation}
    \label{eqn:kappa_0}
\kappa_0\colon I_l^0\ltimes\partial\Omega\to \cU(P_0),\quad [cd^{-1},\chi]\mapsto \rho_0(c)\rho_0(d)^{-1}.
\end{equation}
Moreover, $\ker(\kappa_0)=\partial\Omega$.
\end{lemma}
\begin{proof}
If $[cd^{-1},\chi]=[ab^{-1},\chi]$, then there exists $X\in\chi^{-1}(1)$ such that $X\cap dP\cap bP \neq\emptyset$ and $cd^{-1}$ and $ab^{-1}$ agree on $X\cap dP\cap bP$. Thus, there exists $x,y,z\in P$ such that $z=dx=by$ and $cd^{-1}(z)=ab^{-1}(z)$, i.e., $cx=ay$. Now we have
\[
\rho_0(c)\rho_0(d)^{-1}=\rho_0(c)\rho_0(x)\rho_0(x)^{-1}\rho_0(d)^{-1}=\rho_0(cx)\rho_0(dx)^{-1}=\rho_0(ay)\rho_0(by)^{-1}=\rho_0(a)\rho_0(b)^{-1}.
\]
Thus, we get a well-defined map $\kappa_0$ as in \eqref{eqn:kappa_0}. If $\chi\in\partial\Omega$ and $cd^{-1},ab^{-1}\in I_l^0$, then a short computation yields
\[
[cd^{-1},ab^{-1}.\chi][ab^{-1},\chi]=[cd^{-1}ab^{-1},\chi]=[cx(by)^{-1},\chi],
\]
where $x,y\in P$ satisfy $dx=ay$. We have
\begin{align*}
\kappa_0([cx(by)^{-1},\chi])=\rho_0(cx)\rho_0(by)^{-1}&=\rho_0(c)\rho_0(d)^{-1}\rho_0(d)\rho_0(x)\rho(y)^{-1}\rho_0(a)^{-1}\rho_0(a)\rho_0(b)^{-1}\\
&=\rho_0(c)\rho_0(d)^{-1}\rho_0(dx)\rho_0(ay)^{-1}\rho_0(a)\rho_0(b)^{-1}\\
&=\kappa_0([cd^{-1},ab^{-1}.\chi])\kappa_0([ab^{-1},\chi]).
\end{align*}
Since $\kappa_0$ clearly preserves inverses, this shows that $\kappa_0$ is a homomorphism. That $\kappa_0$ is continuous follows similarly to the proof that $\kappa_\rho$ is continuous in Lemma~\ref{lem:kappa_rho}.

Finally, if $\kappa_0([cd^{-1},\chi])=e$, then $\rho_0(c)=\rho_0(d)$, so that $c=d$ by injectivity of $\rho_0$. Thus, $\ker(\kappa_0)=\partial\Omega$.
\end{proof}

Consider the \emph{boundary core operator algebra} 
\[
\partial\A_{r,0}(P)\coloneq\overline{\alg}(\{1_{[c,\partial\Omega]} :c\in P_0\})\subseteq C_r^*(I_l^0\ltimes\partial\Omega),
\]
and let $C^*(\partial\A_{r,0}(P))$ denote the C*-subalgebra of $C_r^*(I_l^0\ltimes\partial\Omega)$ generated by $\partial\A_{r,0}(P)$. 

\begin{lemma}[cf. {\cite[Remark~4.5]{St22}}]
\label{lem:unitary}
For every $c\in P_0$, we have $\partial\Omega=\partial\Omega(cP)$, so that $1_{[c,\partial\Omega]}$ is a unitary in $C^*(\partial\A_{r,0}(P))$ for every $c\in P_0$.   
\end{lemma}
\begin{proof}
Given $c\in P_0$, we need to show that $\partial\Omega\subseteq\partial\Omega(cP)$. Fix $\chi \in \partial\Omega$. If $0\not\in I_l(P)$, then $\partial\Omega=\{\chi_\infty\}=\partial\Omega(cP)$, where $\chi_\infty$ is the character on $\cJ_P$ given by $\chi(X)=1$ for all $X\in \cJ_P$. 

Now assume that $0\in I_l(P)$. Since $cP\cap dP\neq\emptyset$ for all $d\in P$, and $\cJ_P=\{dP : d\in P\}\cup\{\emptyset\}$, we see that $\{cP\}$ is a (finite) cover for the constructible right ideal $P$. Since $\chi$ is a tight character, $\chi(P)=1$ forces $\chi(cP)=1$, i.e., $\chi\in\partial\Omega(cP)$. 
\end{proof}

There is a canonical action of $\cU(P_0)$ by homeomorphisms of $\partial\Omega$ such that $\rho_0(c)$ is mapped to the homeomorphism $\partial\Omega\to\partial\Omega$ given by $\chi\mapsto c.\chi$. We let $\cU(P_0)\ltimes \partial\Omega$ be the transformation groupoid associated with this action.

\begin{proposition}
\label{prop:transformationgpoid}
    There is an isomorphism of topological groupoids $I_l^0\ltimes\partial\Omega\to \cU(P_0)\ltimes \partial\Omega$ given by $[cd^{-1},\chi]\mapsto (\rho_0(c)\rho_0(d)^{-1},\chi)$.
\end{proposition}
\begin{proof}
For $[cd^{-1},\chi]\in I_l^0\ltimes\partial\Omega$, note that $(\rho_0(c)\rho_0(d)^{-1},\chi)=(\kappa_0(cd^{-1}),\mfs((\rho_0(c)\rho_0(d)^{-1},\chi)))$, where $\kappa_0$ is the cocycle from Lemma~\ref{lem:kappa_0}. This shows that the map is well-defined; a short calculation shows that this map is indeed an isomorphism of topological groupoids. 
\end{proof}

As a straightforward consequence of Proposition~\ref{prop:transformationgpoid}, we have the canonical *-isomorphism between reduced groupoid C*-algebras $C_r^*(I_l^0\ltimes\partial\Omega)\cong C(\partial \Omega)\rtimes^r\cU(P_0)$. 
Under this isomorphism, $C^*(\partial\A_{r,0}(P))$ is carried onto the (copy of) $C_r^*(\cU(P_0))$ inside $C(\partial \Omega)\rtimes^r\cU(P_0)$, where we have used here compactness of $\partial\Omega$. Thus, $C^*(\partial\A_{r,0}(P))$ carries a canonical normal coaction $\delta_0\colon \cU(P_0)\coacts C^*(\partial\A_{r,0}(P))$ that is determined by $\delta_0(1_{[c,\partial\Omega]})=1_{[c,\partial\Omega]}\otimes u_{\rho_0(c)}$ for all $c\in P_0$. 
In particular, $\delta_0$ restricts to a normal coaction on $\partial\A_{r,0}(P)$, which we also denote by $\delta_0$.
The faithful conditional expectation on $C^*(\partial\A_{r,0}(P))$ associated with $\delta_0$ is then the tracial state obtained from the canonical trace on $C_r^*(\cU(P_0))$ via the isomorphism $C^*(\partial\A_{r,0}(P))\cong C_r^*(\cU(P_0))$.

In order to identify the C*-envelope, we shall now utilize a recent result by Starling \cite[Theorem~4.1]{St22} (see \cite[Theorem~2.1]{BS24} for the precise formulation and proof), stating that for a right LCM monoid $P$, a representation $\pi$ of $C_r^*(I_l\ltimes\partial\Omega)$ is injective if and only if it is injective on $C^*(\partial\A_{r,0}(P))$. 

The following resolves Li's Question~\ref{ques:Li} for all cancellative right LCM monoids, including many aforementioned non group-embeddable examples.

\begin{theorem}
\label{thm:lcm-monoids}
    Let $P$ be a cancellative right LCM monoid. Then, the associated groupoid $I_l\ltimes\Omega$ is Hausdorff, and the map $\pi_\env\colon C_r^*(I_l\ltimes\partial\Omega)\to C_\env^*(\A_r(P))$ from \eqref{eqn:pi} is a *-isomorphism. That is, the C*-envelope $C^*_\env(\A_r(P))$ coincides with the boundary quotient C*-algebra $C_r^*(I_l\ltimes\partial\Omega)$.
\end{theorem}
\begin{proof}
Throughout this proof, we identify $\A_r(P)$ with its image in $C_\env^*(\A_r(P))$ under the canonical completely isometric embedding.
Since $P$ is right LCM,  the groupoid $I_l\ltimes\Omega$ is Hausdorff by \cite[Corollary~3.24]{Nor14} and \cite[Corollary~10.9]{Exel:comb}.
By \cite[Theorem~4.1]{St22} (see also \cite[Theorem~2.1]{BS24}) the canonical map $\pi_\env$ is injective if and only if the restriction  $\varphi\coloneq\pi_\env\vert_{C^*(\partial\A_{r,0}(P))}$ is injective. 
Put
\[
\A_{r,0}(P)\coloneq \overline{\alg}(\{\lambda_P(c) : c\in P_0\})\subseteq \A_r(P).
\]
By Theorem \ref{thm:ciAr}, we have that $\pi_\env\vert_{\partial\A_r(P)}\colon \partial\A_r(P)\to \A_r(P)$ is completely isometric and maps generators to generators, 
so that $\pi_\env$ restricts to a completely isometric isomorphism from $\partial\A_{r,0}(P)$ onto $\A_{r,0}(P)$.
Recall that $\partial \A_{r,0}(P)$ admits a normal coaction $\delta_0$ of $\mathcal{U}(P_0)$,
so there is a normal coaction $\eps$ of $\cU(P_0)$ on $\A_{r,0}(P)$ given by $\eps\coloneq (\pi_\env\otimes\id_{C^*(\cU(P_0))})\circ\delta_0\circ(\pi_\env\vert_{\partial\A_{r,0}(P)})^{-1}$. 
By Theorem~\ref{thm:main}, this extends to a compatible normal coaction $\eps_\env\colon \cU(P_0)\coacts C_\env^*(\A_{r,0}(P))$. 

Let $C^*(\A_{r,0}(P))$ be the C*-subalgebra of $C_\env^*(\A_r(P))$ generated by $\A_{r,0}(P)$, 
and consider the canonical surjective *-homomorphism $\psi\colon C^*(\A_{r,0}(P))\to C_\env^*(\A_{r,0}(P))$ extending the identity map on $\A_{r,0}(P)$. 
We obtain the following commuting diagram:

\[\begin{tikzcd}
 C^*(\partial\A_{r,0}(P)) \arrow[d,"E^{\delta_0}"]\arrow[r,"\varphi"] & C^*(\A_{r,0}(P))  \arrow[r,"\psi"]& C_\env^*(\A_{r,0}(P))\arrow[d,"E^{\eps_\env}"]\\
  C^*(\partial\A_{r,0}(P))^{\delta_0}_e \arrow[rr,"\psi\circ\varphi\vert_{C^*(\partial\A_{r,0}(P))^{\delta_0}_e}"] & & C_\env^*(\A_{r,0}(P))^{\eps_\env}_e,
\end{tikzcd}\]
where $E^{\delta_0}$ and $E^{\eps_\env}$ are the faithful conditional expectations associated with $\delta_0$ and $\eps_\env$, respectively. 
Since $\pi_\env$ is injective on $C^*(\partial\A_{r,0}(P))^{\delta_0}_e=\Cz1_{\partial\Omega}$,
commutativity of the above diagram and faithfulness of $E^{\delta_0}$ implies that $\psi\circ\varphi$ is injective.
Therefore, $\varphi$ is injective. 
\end{proof}

\end{document}